\documentclass{amsart}
\usepackage[main=english]{babel}
\usepackage[T1]{fontenc}
\usepackage{graphicx}
\usepackage{amscd}
\usepackage{amsmath}
\usepackage{amsfonts}
\usepackage{amssymb}
\usepackage{bbm}
\usepackage{setspace}
\usepackage{enumerate} 
\usepackage{fixme}
\usepackage{color}
\usepackage{url}
\usepackage{amsthm}
\usepackage{bm}
\usepackage{xy}
\usepackage{enumitem}
\usepackage{comment}
\usepackage{xcolor}
\usepackage{mathrsfs}
\usepackage{dsfont}

\RequirePackage{colortbl}
\definecolor{tscolor}{rgb}{1.0,0.6,0.0}
\usepackage[ulem=normalem,commandnameprefix=ifneeded]{changes}      
\definechangesauthor[name = Thorsten Schmidt, color=tscolor]{TS}  
\definechangesauthor[name = Irene Klein, color=mypurple]{IK}

\usepackage{natbib}

\theoremstyle{plain}
\newtheorem{theorem}{Theorem}[section]

\newtheorem{lemma}[theorem]{Lemma}
\newtheorem{notation}[theorem]{Notation}

\newtheorem{definition}[theorem]{Definition}

\newtheorem{assumption}[theorem]{Assumption}
\theoremstyle{remark}
\newtheorem{remark}[theorem]{Remark}
\newtheorem{example}[theorem]{Example}

\numberwithin{equation}{section}

\newcommand{\rsto}{]\!\kern-1.8pt ]}
\newcommand{\lsto}{[\!\kern-1.7pt [}

\numberwithin{equation}{section}

\newcommand{\RR}{\mathbb{R}}

\newcommand{\cP}{\mathcal{P}}

\newcommand{\LL}{\mathbb{L}^{\infty}}

\renewcommand{\rho}{\varrho}

\newcommand\ep{\varepsilon}

\newcommand{\ind}{{\mathds 1}}

\newcommand{\newinf}{\mathop{\mathrm{inf}\vphantom{\mathrm{sup}}}}

\usepackage[colorlinks,urlcolor=red,citecolor=blue,linkcolor=red]{hyperref}

\newcommand{\pp}[0]{\ensuremath \mathcal{P}}

\newcommand{\rrr}[0]{\ensuremath \mathcal{R}}

\newtheorem{mass}[theorem]{Assumption}

\definecolor{mypurple}{rgb}{0.5,0,98}
\definecolor{georggreen}{rgb}{0.02,0.7,0.2}

\begin{document}

\title[Robust large financial markets]{Quantitative Halmos-Savage Theorems and robust large financial markets}

\begin{abstract}

We establish a quantitative version of the classical Halmos-Savage Theorem for convex, potentially non-dominated sets of probability measures and its dual counterpart,  generalizing previous quantitative versions. 
These results are then used to derive robust versions of the fundamental theorem of asset pricing (FTAP) in large financial markets in a one-period setting, characterizing the absence of arbitrage under Knightian uncertainty. To this end, we consider robust formulations of no asymptotic arbitrage of first kind (NAA1), which is the large market analogue of ``No unbounded profit with bounded risk'' (NUPBR), as well as no asymptotic arbitrage of second kind (NAA2). Finally, we characterize asymptotic arbitrage of first and second kind in 
the robust one-period binomial model in terms of the model parameters.
\end{abstract}

\thanks{This research was funded in whole or in part by the Austrian Science Fund (FWF) Y 1235.\\
Support by the DFG with grant SCHM 2160-13 is gratefully acknowledged, too. Georg Köstenberger is part of Research Unit 5381 of the German Research Foundation and is supported by the Austrian Science Fund (FWF) I 5485-N}
\keywords{Halmos-Savage Theorem, NAA1 and NAA2, robust financial market, large financial markets, discrete time, asymptotic arbitrage.}
\subjclass[2000]{91G15, 60A10, 91B70, 91G80, 62G35}

\author[Cuchiero, Klein, K\"ostenberger, Schmidt]{Christa Cuchiero,  Irene Klein, Georg K\"ostenberger, Thorsten Schmidt}
\address{}
\maketitle

\section{Introduction}
Modeling financial markets as a dynamic probabilistic 
system is often approached via choosing an appropriate stochastic model. The choice of a single model -- be it by statistical estimation, machine learning methods, or calibration to market data -- is subtle and often erroneous. It has therefore become a core topic in modern approaches to financial markets to replace a single model $P$ by a suitable class of models $\cP$ and develop methodologies which perform well for the whole class. The set $\cP$ is often called \emph{uncertainty set}, since this idea can be traced back to Frank Knight's 
famous notion of \emph{uncertainty} introduced in \cite{knight1921risk}.

In this work we are concerned with large financial markets in a one-period setting 
under uncertainty and are interested in developing fundamental theorems which characterize the absence of asymptotic arbitrages of  first and second kind, extending existing results to the setting with uncertainty. More precisely, we understand a large financial market under uncertainty as a sequence of finite-dimensional financial markets under uncertainty, i.e., via a sequence of uncertainty sets $(\cP^n)_{n \ge 1}$. For the  finite-dimensional market we assume absence of arbitrage as  developed in \cite{BN:15}. 
We focus on the one-period case, the multi-period setting is the subject of future research.
Following \cite{KK:94}, we introduce robust formulations of arbitrages of  first and second kind.
Extending the classical formulations, we require in addition that there exists a (sub)sequence of probability measures $(P^k)_{k \ge 1}$ from the associated uncertainty sets $\cP^{n_k}$, such that $P^k(X^k_{1} \ge \alpha) \ge \alpha$ for some $\alpha >0$ for an  asymptotic arbitrage of first kind (AA1), and $P^k(X^k_{1} \ge \alpha) \to 1$ for an  asymptotic arbitrage of second kind (AA2). 
Here, $X^k_{1}$ are the terminal values of  sequences of the suitably bounded (i.e., bounded from below $\mathcal{P}^{n_k}$-quasi-surely) portfolios considered in the arbitrage strategy. 
In the case of AA1 the lower bound is vanishing with $k\to\infty$, whereas in the case of AA2 the bound is a uniform constant for all $k$. 
Under a technical condition on the set of probability measures $\mathcal{P}^n$, Theorem \ref{lfmftap} states that the robust version of NAA1 is equivalent to the following: for each sequence $(P^n)_{n \ge 1}$ of probability measures in $(\cP^n)_{n \ge 1}$ there exists a sequence of martingale measures $(Q^n)_{n \ge 1}$, suitably dominated, such that $(P^n)_{n \ge 1}$ is contiguous with respect to $(Q^n)_{n \ge 1}$. Theorem \ref{lfmftap2} shows that NAA2 is equivalent to weak contiguity of the set of suitably dominated martingale measures to every sequence $(P^n)_{n \ge 1}$ in $(\cP^n)_{n \ge 1}$ (again under the same technical condition on $\mathcal{P}^n$).
We then apply these results to a sequence  of one-period robust binomial models as studied in \cite{blanch_carassus}, \cite{liebrich2022} and \cite{klein-koestenberger-robust-duality}. In this setup it is possible to fully characterize  NAA1 and NAA2 in terms of the involved parameters, see Theorems \ref{thm_U0_finite} - \ref{thm_U0_infinite_AA2}. We also compare these results with the no arbitrage condition of \cite{BN:15} for the corresponding robust limit model. We show that one can construct examples with no arbitrage in the limit model while there exists some AA1 or some AA2, and also vice versa.

The main tool which we develop to prove the two fundamental theorems is a quantitative version (and its dual counterpart) of the Halmos-Savage Theorem for a convex set of probability measures $\cP$. The well-known theorem of Halmos and Savage 
proves the following: if $\mathcal{Q}$ is a 
set of probability measures on $(\Omega,\mathcal{F})$,  which is closed under countable convex combinations and which is dominated by some probability measure $P$, such that for each measurable $A$ with $P(A)>0$ there exists $Q\in\mathcal{Q} $ with $Q(A)>0$,  then
 there exists $Q\in \mathcal{Q}$ with full 
support, i.e., $Q\sim P$, see
\cite{HS:49}. In \cite{KS:96b} a quantitative and a dual quantitative version of this result was developed and applied to large financial markets. Here, we will further generalize these quantitative versions of  \cite{KS:96b} by replacing the fixed measure $P$ by a convex set of potentially non-dominated probability measures $\mathcal{P}$ and assume that $\mathcal{Q}\lll\mathcal{P}$ which is a natural extension, meaning that the set $\mathcal{P}$ ``dominates'' the set $\mathcal{Q}$ in the sense of Definition \ref{lll}. 
We additionally need that $\mathcal{P}$-quasi-surely bounded random variables (denoted by $\mathbb{L}^{\infty}$) are isomorphic to the dual space 
of all finite signed measures that are absolutely continuous with respect to at least one $P \in \mathcal{P}$.
This allows us to get a robust analogue of the classical $L^1$-$L^{\infty}$ duality and  enables us to apply our theorems to \emph{robust} large financial markets. We give precise formulations of these results in Section \ref{secHS}. The existing versions of  \cite{KS:96b} in the non-robust setting are recovered from the special case $\cP=\{P\}$. The key step for the proofs is to find an appropriate topological setting with a dual pair of normed spaces. With some additional effort, Sion's Minimax theorem can be applied, which then leads to Theorem~\ref{HS} and Theorem~\ref{dualHS}.

\subsection*{Related literature}

In modern approaches to financial markets, it is of fundamental importance to assess model risk in an appropriate way. Approaches which address this are typically called \emph{robust} and a large variety exist,  see, e.g., 
\cite{Den-Mar}, \cite{Peng-gexp},   \cite{Den-Ker}, \cite{Hu-Peng}, \cite{Tev-Tor},  \cite{Bia-Bou},  \cite{Lar-Acc}, and \cite{Neu-Nut}. Additionally, approaches relying on techniques of stochastic control and stochastic backward differential equations have been considered, see, e.g., \cite{Peng2019}, \cite{neufeld2017nonlinear},   \cite{fns_2019}, \cite{geuchen2022affine}, \cite{CriensNiemann},  and \cite{criens2023jumps} for the case with jumps. In this context, 
various notions  of arbitrage under uncertainty have been considered in the setting of Knightian uncertainty, see, e.g., \cite{BN:15}, \cite{blanch_carassus}, \cite{carassus2025quasi}, and \cite{bayraktar2017arbitrage}. Based on a pathwise approach, as, e.g., in  \cite{davis2007range} or \cite{acciaio2016model}, \cite{burzoni2016universal} have formulated a more general scenario-based approach. Other approaches like  \cite{carassus2022pricing} aim at weakest possible assumptions on the financial model. 
In \cite{obloj2021unified} these seemingly different approaches were unified under weak assumptions.

To build portfolio wealth processes in such model-free setups, \cite{PP:16} have developed a pathwise stochastic integration theory, where the second order lift of rough path theory is no longer needed, thus allowing for a natural financial interpretation. This approach has for instance been taken up in \cite{allan2023model, ALP:24}.
Asymptotic arbitrages have also been considered in relation to large insurance portfolios, see for example \cite{artzner2024insurance} and \cite{oberpriller2024robust}.

For some literature on large financial markets we refer, e.g., to \cite{KK:94}, where asymptotic arbitrage of first and second kind was introduced, to \cite{KK:98}, \cite{KS:96}, where the theory was further developed, to \cite{K:00}, \cite{CKT:14} for some very general versions of the fundamental theorem of asset pricing in continuous time based on the notion of asymptotic free lunch, and to \cite{R:03} for a result in discrete time.

\subsection*{Structure of the paper}

In Section \ref{secHS}, we provide the robust versions of the Halmos-Savage Theorem and an overview of the idea of the proofs, while Section~\ref{ProofHS} contains the detailed proofs.  In Section~\ref{sec:RLFM} we develop the fundamental theorems for robust large financial markets, Theorems \ref{lfmftap} and \ref{lfmftap2}.
In Section~\ref{App_robust_bin} we study sequences of robust one-period binomial models and their asymptotic arbitrage opportunities.

\section{A quantitative and a quantitative dual Halmos-Savage Theorem}\label{secHS}
Let $(\Omega,\mathcal{F})$ be a measure space and $\mathcal{P}$ a convex set of probability measures on $(\Omega,\mathcal{F})$.
In this section, we generalize the quantitative version and the dual quantitative version of the Halmos-Savage Theorem of \cite{KS:96b} to the setting of an convex, potentially non-dominated set $\mathcal{P}$ instead of a fixed probability measure $P$. The original versions of \cite{KS:96b} can be considered as special cases of our theorems with $\mathcal{P}=\{P\}$  being a singleton.

\begin{definition}\label{lll} For any set $\mathcal{R}$ of probability measures on $(\Omega,\mathcal{F})$ we say that $\mathcal{R}\lll\mathcal{P}$ if for every $R\in \mathcal{R}$ there exists $P\in \mathcal{P}$ such that $R\ll P$.
\end{definition}
We will use the notion quasi-surely (q.s. for short) as an appropriate generalization of almost surely to the setting with a set $\mathcal{P}$ of probability measures.
\begin{definition}
A set $A'\subset \Omega$ is a \emph{polar set} if 
 $A'\subset A$ for some $A \in\mathcal{F}$ with $P(A)=0$ for all $P\in\mathcal{P}$.
A property holds $\mathcal{P}$-quasi-surely if it holds outside a polar set.
\end{definition}

The statement of the robust Halmos-Savage Theorems requires a robust notion of the classical $L^{1}$-$L^{\infty}$ duality. 
Let $\mathcal{L}^0(\Omega,\mathcal{F})$ be the space of all measurable functions $f:(\Omega,\mathcal{F})\to(\mathbb{R},\mathcal{B})$ where $\mathcal{B}$ is the Borel $\sigma$-algebra on $\mathbb{R}$. 
Following \cite{DHuPe:2011} we define
\begin{align}\label{Linf}
  \mathcal{L}^{\infty} &=\{X\in \mathcal{L}^0(\Omega,\mathcal{F}): \exists\text{ a constant $M$, such that $|X|\leq M$, $\mathcal{P}$-q.s.}\}, \quad \text{and} \nonumber \\
\LL &= \mathcal{L}^{\infty}/ \mathcal{N},
\end{align}
where $\mathcal{N}=\{X\in \mathcal{L}^0(\Omega,\mathcal{F}): X=0,\ \mathcal{P} \text{-q.s.}\}$. Note that $\mathbb{L}^{\infty}$ is a Banach space with respect to the following norm
\begin{equation}\label{inftynorm}
\|X\|_{\mathbb{L}^\infty}=\inf\{M\geq 0: |X|\leq M, \mathcal{P}\text{-q.s.}\}.
\end{equation}
For any normed space $X$, we write $X'$ for its dual.
We denote with $E$ the space of all finite signed measures on $\mathcal{F}$ which are absolutely continuous with respect to at least one $P\in \mathcal{P}$. 
This space can be equipped with the total variation norm. 
Note that in the classical case $\mathcal{P} = \{P\}$, the space
$E$ is isometrically isomorphic to $L^{1}(P)$, and hence $E'$ is isometrically isomorphic to $L^{\infty}(P)$. Therefore, the space $E$ can be thought of as a robust version of an $L^{1}$ space. 
We will see in Section~\ref{ProofHS} that the space $\LL$ embeds isometrically into $E'$ via the map $l: \LL \to E', h\mapsto l_{h}$, where
\begin{equation*}
  l_{h}(\mu) = \int_{\Omega} h d \mu, \quad \mu\in E.
\end{equation*}
If $\mathcal{P} = \{P\}$, this is the usual surjective embedding of $L^{\infty}(P)$ into $L^{1}(P)'$.
It turns out, that $l(\LL)$ is always weak* dense in $E'$ (see Lemma \ref{lemma:robust-L-space-is-dense}). 
However, $l(\LL)$ may not be weak* closed. 
 In \cite{klein-koestenberger-robust-duality} it was shown that many models (including, in particular, robust versions of the binomial model, and of the Black-Scholes model) admit a canonical extension, on which $l(\LL) = E'$.
For a full characterization of this property in a slightly different setting, we refer the reader to \cite{liebrich2022}. 
We can now formulate the quantitative versions of the Halmos-Savage Theorem for convex sets of probability measures.

\begin{theorem}[Quantitative Halmos-Savage for a convex set of probabilities]\label{HS}
Assume that $l(\LL) = E'$.
Let $\mathcal{P}$ and  $\mathcal{Q}$ be convex sets of probability measures on $(\Omega,\mathcal{F})$ such that $\mathcal{Q}\lll\mathcal{P}$. 
Assume that for some fixed $\ep>0$ and $\delta>0$ the following holds:
for each $A\in\mathcal{F}$ such that there exists $P\in \mathcal{P}$ with $P(A)\geq\ep$ there exists $Q\in \mathcal{Q}$ with $Q(A)\geq\delta$. Then for each $P\in \mathcal{P}$ there exists $Q\in\mathcal{Q} $ such that for each $A\in\mathcal{F}$ with $P(A)\geq 2\ep$ we have that $Q(A)\geq\frac{\ep\delta}{2}$.
\end{theorem}

Note that Proposition~1.3 of  \cite{KS:96b} is a special case of Theorem~\ref{HS} with $\mathcal{P}=\{P\}$ and $\mathcal{Q}=M$. Moreover, we can also generalize the dual result, that is, Proposition~1.5 of  \cite{KS:96b}: 

\begin{theorem}[Dual quantitative Halmos-Savage for a convex set of probabilities]\label{dualHS}
Assume that $l(\LL) = E'$.
Let $\mathcal{P}$ and  $\mathcal{Q}$ be convex sets of probability measures on $(\Omega,\mathcal{F})$ such that $\mathcal{Q}\lll\mathcal{P}$.
Assume that for some fixed $\ep>0$ and $\delta>0$ the following holds:
for each $A\in\mathcal{F}$ such that there exists $P\in \mathcal{P}$ with $P(A)<\delta$ there exists $Q\in \mathcal{Q}$ with $Q(A)<\ep$. Then for each $P\in \mathcal{P}$ there exists $Q\in\mathcal{Q} $ such that for each $A\in\mathcal{F}$ with $P(A)<\ep\delta$ we have that $Q(A)<2\ep$.
\end{theorem}

The proofs of the above theorems are inspired by the proofs of the results in \cite{KS:96b}, in particular by the modified proof that was given 
in \cite{K:thesis}. However, the difficulty lies in finding  an appropriate topological setting tailored to the current robust situation where we deal with convex, potentially non-dominated sets of probability measures.  The remainder of this section will be devoted to the basic ideas of the proofs. The details will be presented in Section~\ref{ProofHS}.

Similarly as in  \cite{K:thesis}, our proof starts with the definition of the following sets.
\smallskip

\begin{definition}\label{the_sets_D} Fix $P\in\mathcal{P}$ and $\ep,\delta>0$. Define the following subsets of $\mathbb{L}^{\infty}$:
\begin{enumerate}
\item
$D^{\ep,P}=\{h\in\mathbb{L}^{\infty} \colon  0\leq h \leq 1\  \mathcal{P}\text{-q.s. and }E_P[h]\geq 2\ep\}$, and 
\item $\widetilde{D}^{\ep\delta, P}=\{h\in\mathbb{L}^{\infty} \colon  0\leq h \leq 1\  \mathcal{P}\text{-q.s. and }E_P[h]\leq \ep\delta\}$.
\end{enumerate}
\end{definition}
Observe that the sets $ D^{\ep,P}$ and $\widetilde{D}^{\ep\delta, P}$ are convex subsets of $\LL$. Indeed, the property of being quasi-surely bounded below by $0$ and above by $1$ obviously carries over to a convex combination. The property of the expected value under $P$ being $\geq 2\ep$ (or $\leq\ep\delta$) only has to hold with respect to the fixed measure $P$. Hence, the property carries over to any convex combination.
\smallskip

  Let us briefly outline the proofs of Theorems \ref{HS} and \ref{dualHS}. 
  It suffices to focus on Theorem \ref{HS}, as Theorem \ref{dualHS} can be proved using almost identical arguments.
  In Lemma \ref{basicLem} we establish that the assumptions of Theorem \ref{HS} imply
  \begin{equation*}
\newinf_{h\in D^{\ep, P}}\sup_{Q\in \mathcal{Q}}E_Q[h]\geq\ep\delta.
\end{equation*}
If we could exchange the infimum with the supremum, i.e., if
\begin{equation*}
  \sup_{Q\in \mathcal{Q}}\newinf_{h\in D^{\ep, P}}E_Q[h] = \newinf_{h\in D^{\ep, P}}\sup_{Q\in \mathcal{Q}}E_Q[h]\geq\ep\delta,
\end{equation*}
then Theorem \ref{HS} follows almost immediately (see page \pageref{proof:HS} for details). 
The following Minimax Theorem, see \cite{Sion}, addresses precisely this issue:

\begin{theorem}[Sion's 
Minimax Theorem]\label{thm:sion}
  Let $T$ and $S$ be topological vector spaces, $X \subseteq T$ convex and $Y \subseteq S$ convex and compact. Assume that $f: X \times Y \to \mathbb{R}$ satisfies
  \begin{enumerate}
  \item $f(\cdot, y)$ is upper semicontinuous and quasi-concave on $X$, for every fixed $y\in Y$, and 
  \item $f(x, \cdot)$ is lower semicontinuous and quasi-convex on $Y$, for every fixed $x\in X$.
  \end{enumerate}
  Then we have
  \begin{equation*}
 \sup _{x\in X}\newinf_{y\in Y}f(x,y)=\newinf _{y\in Y}\sup _{x\in X}f(x,y).
  \end{equation*}

\end{theorem}
We want to apply this theorem with $X= \mathcal{Q}$, $Y = D^{\varepsilon,P}$, and $f(Q,h) = E_{Q}[h]$. 
Using the embedding $l: \LL\to E'$, we can identify $D^{\varepsilon,P}$ with its image $l(D^{\varepsilon,P})$.
Note that $\mathcal{Q}$ and $l(D^{\varepsilon,P})$ are convex, and that the function $f(Q,h) = E_{Q}[h]$ is linear in both its arguments, and hence concave in $Q$ and convex in $h$. 
The continuity of $f$ is a direct consequence of the topological setup, see Section \ref{ProofHS}. 
The assumption $l(\LL) = E'$ guarantees that $l(D^{\varepsilon,P})$ is a weak* closed subset of $E'$.
The Banach-Alaoglu Theorem implies that $l(D^{\varepsilon,P})$ is actually weak* compact. 
Hence, the assumptions of Sion's Minimax Theorem are satisfied, and we can conclude the proof of Theorem \ref{HS} by the argument above.

\section{Proofs  of Theorem~\ref{HS} and Theorem \ref{dualHS}}\label{ProofHS}

Let us start by recalling some properties of the total variation norm.
For $\mu\in E$ (with $\mu\ll P'$) let $\mu=\mu^+-\mu^-$ be the Hahn-Jordan decomposition of $\mu$ and let $\Omega=\Omega^+\cup\Omega^-$ be the corresponding Hahn-decomposition of $\Omega$. 
The (non-negative) measures $\mu^+$ and $\mu^-$ are defined as follows: $\mu^+(A)=\mu(A\cap \Omega^+)$ and $\mu^-(A)=-\mu(A\cap \Omega^-)$, for all $A\in\mathcal{F}$. 
Clearly, $\mu^+(\Omega)=\mu(\Omega^+)$, $\mu^-(\Omega)=\mu(\Omega^-)$, and $\mu(\Omega)=\mu^+(\Omega^+)-\mu^-(\Omega^-)$.
Moreover,   as $\mu^+, \mu^-\ll\mu\ll P'$, for some $P'\in\mathcal{P}$, we have that $\mu^+, \mu^-\in E $.
The total variation norm on $E$ denoted by $\|\cdot\|$ is given by 
\begin{equation}\| \mu\|= \mu^+(\Omega^+)+\mu^-(\Omega^-).\label{totvarn}\end{equation}
We will give the short proof that $E$ is a normed vector space where the zero vector in $E$ is the zero measure $\mu_0$ with $\| \mu_0\|=0$, i.e., for its Hahn-Jordan decomposition we have that $\mu_0(\Omega^+)=\mu_0(\Omega^-)=0$. In particular, $\mu_0\ll P'$ for all $P'\in\mathcal{P}$.

\begin{lemma}\label{lemma:E-is-normed-space} If $\mathcal{P}$ is convex, then, $E$ is a normed vector space.
  If $\mathcal{P}$ is closed under countable convex combinations, then $E$ is a Banach space. 
\end{lemma}

\begin{proof}
Let $\mu_1$, $\mu_2$ be in $E$ and $c\in\mathbb{R}$.  Then $\mu_1+c\mu_2$ obviously is a finite signed measure. Moreover, by assumption, $\mu_i\ll P_i$, $i=1,2$, where $P_1,P_2\in\mathcal{P}$. As by assumption $\mathcal{P}$ is convex, $P'=\frac{P_1+P_2}{2}\in\mathcal{P}$ and $\mu_i\ll P'$ for $i=1,2$. Hence $\mu_1+c\mu_2\ll P'$. The total variation norm is a norm.

Now, observe that if $\mathcal{P}$ is not only convex but also contains all countable convex combination of measures $P_n\in\mathcal{P}$, then $E$ is a Banach space. 
Indeed, let $\mu_n\to\mu$ in total variation norm. Obviously, $\mu$ is a finite signed measure. 
By assumption, for each $\mu_n$ there exists $P_n\in\cP$ such that $\mu_n\ll P_n$. Define $P=\sum_{n=1}^{\infty}2^{-n}P_n\in\mathcal{P}$. 
Then $\mu_n\ll P$, for all $n$, and the Vitali-Hahn-Saks Theorem implies $\mu\ll P$.
\end{proof}

Next, we would like to define a \emph{compatible} topology $\sigma(\LL, E)$ on $\mathbb{L}^{\infty}$. 
Denote by $B: E\times \LL \to \mathbb{R}$ the bilinear form given by
\begin{equation}\label{bilinearform}
    B(\mu, h) = \int_{\Omega}h d \mu.
\end{equation} 
$B$ separates the points on $E$ and $\LL$, i.e., 
\begin{enumerate}
\item $B(\mu, h) = 0$ for all $h\in \LL$ implies $\mu = 0$, and 
\item $B(\mu, h) = 0$ for all $\mu\in E$ implies $h = 0$. 
\end{enumerate}
Hence, the family of seminorms $\{h \mapsto |B(\mu, h)|: \mu \in E\}$ on $\LL$ defines a locally convex Hausdorff topology on $\LL$, which we denote by $\sigma(\LL, E)$. 
Similarly, the weak* topology on $E'$, denoted with $\sigma(E',E)$, is the locally convex Hausdorff topology induced by the seminorms $\{ y \mapsto |y(\mu)|: \mu\in E\}$ on the dual space $E'$.

Recall that the dual space $E'$  of the normed vector space $(E,\|\cdot\|)$ is a Banach space, where the norm of $y\in E'$ is given by

\begin{equation}\label{norm_dual}
  \|y \|_{E'}=\sup_{\substack{\mu\in E\\ \| \mu\|\leq 1}}|y(\mu)|.
\end{equation} 
Furthermore, we denote with
\begin{equation}\label{bilinEE'}
\langle \mu,y \rangle =y(\mu)
\end{equation} the bilinear form of the dual pair $(E,E')$.

\begin{definition}\label{embedding} Let $E^{*}$ be the algebraic dual of $E$, that is, the space of all linear functionals on $E$.
Define the map $l: \LL\to E^{*}$ as follows: for $h\in\LL$ let $l(h)=l_h$ where
\begin{equation}\label{linFunc} l_h(\mu)=B(\mu,h)=\int_{\Omega} hd\mu, \text{ for all $\mu\in E$}. \end{equation}
\end{definition}

\begin{lemma}\label{iso_embed}  Let $l(\LL)=\{l(h):h\in \LL\}$. Then   $l(\LL)\subset E'$, that means each $l_h$ as in (\ref{linFunc}) is a continuous linear functional on $E$. Moreover the embedding $l$ of Definition~\ref{embedding} is isometric, i.e., $$\|l_h\|_{E'}=\|h\|_{\LL}.$$
\end{lemma}

\begin{proof}
Let $h\in\LL$ and define $l_h$ as above.  We only have to show that $l_h: E\to\mathbb{R}$ is continuous.
Note that for every $\mu\in E$ there exists $P'\in\mathcal{P}$ with $\mu\ll P'$. Let again $\mu=\mu^+-\mu^-$ be the Hahn-Jordan decomposition of $\mu$. Then $\mu^+\ll P'$
and $\mu^- \ll P'$. As $h\in\LL$ we have that $P'(|h|\leq M)=1$, where $M=\| h\|_{\LL}$. Hence $\mu^+(|h|>M)=\mu^-(|h|>M)=0$. Therefore we get, for every $\mu\in E$,
\begin{align*}
|l_h(\mu)| &\leq \int |h|d\mu^++  \int |h|d\mu^-\\
&\leq M(\mu^+(\Omega^+)+\mu^-(\Omega^-))=\|h\|_{\LL}\cdot\| \mu\|.
\end{align*}
This implies that $l_h$ is continuous and $\| l_h\|_{E'}\leq \|h\|_{\LL}$.

To show that $l$ is an isometry, note that for each $\ep>0$, there exists $P_{\ep}\in\cP$ with $P_{\ep}(A_{\ep})>0$, where $A_{\ep}:=\{|h|>M-\ep\}$. 
Define  $A_{\ep}^+=A_{\ep}\cap\{h\geq 0\}$ and $A_{\ep}^-=A_{\ep}\cap\{h<0\}$ and define, for each $\ep>0$,  a finite signed measure $\mu_{\ep}\in E$ by its Radon Nikodym derivative
$$\frac{d\mu_{\ep}}{dP_{\ep}}=\frac{1}{P_{\ep}(A_{\ep})}\left(\ind_{A_{\ep}^+}-\ind_{A_{\ep}^-}\right).$$
Then, clearly, the Hahn decomposition of $\mu_{\ep}$ is equal to $\Omega=A_{\ep}^+\cup A_{\ep}^-$ and
$$
\| \mu_{{\ep}}\| =|\mu_{\ep}(A_{\ep}^+)|+|\mu_{\ep}(A_{\ep}^-)| =\frac{P_{\ep}(A_{\ep}^+)+P_{\ep}(A_{\ep}^-)}{P_{\ep}(A_{\ep})}=1.$$
Moreover,
\begin{align*}
|l_h(\mu_{\ep})|&=\left |\int h d\mu_{\ep}\right | =\frac1{P_{\ep}(A_{\ep})}\left|\int (h\ind_{A_{\ep}^+}- h\ind_{A_{\ep}^-})dP_{\ep}\right| \\
&=\frac1{P_{\ep}(A_{\ep})}\int |h|\ind_{A_{\ep}}dP_{\ep}>M-\ep.\end{align*}
This implies
$$\|l_h\|_{E'}=\sup_{\substack{\mu\in E\\ \| \mu\|\leq 1}}|l_h(\mu)|\geq \sup_{\ep>0}|l(\mu_{\ep})|\geq M=\|h\|_{\LL},$$
and hence $\|l_h\|_{E'}=\|h\|_{\LL}$.
\end{proof}

Before we start with the proofs of Theorem \ref{HS} and Theorem \ref{dualHS}, we would like to briefly explore and justify the condition $l(\LL) = E'$.
It turns out that this assumption is satisfied (up to an extension of the underlying probability space) for a large class of models, see \cite{klein-koestenberger-robust-duality}.
Moreover, $l(\LL)$ is $\sigma(E',E)$-dense in $E'$ for any convex set of probability measures $\mathcal{P}$. 
In other words, the condition $l(\LL) = E'$ is equivalent to $l(\LL)$ being weak$^*$ closed.

\begin{lemma}\label{lemma:robust-L-space-is-dense} Suppose $l(\mathbb{L}^{\infty})$ is $\sigma(E',E)$ closed. Then $l(\mathbb{L}^{\infty})=E'$.\end{lemma}
\begin{proof}
  The assumption that $l(\mathbb{L}^{\infty})$ is $\sigma(E',E)$ closed implies that  
  $l(\mathbb{L}^{\infty})$ is a weak* closed subspace of the locally convex (with respect to the weak* topology) space $E'$. Suppose there would exist $y'\in E'\setminus
l(\mathbb{L}^{\infty})$. By
Hahn-Banach there exists a weak*-continuous linear functional on $E'$, that is, a $\mu\in E$, such that
\begin{enumerate}
\item $\langle \mu, y' \rangle = y'(\mu) = 1$, and 
\item $\langle \mu, l_{h} \rangle = l_{h}(\mu) = B(\mu,h) = 0$, for all $h\in \LL$. 
\end{enumerate}
Point (ii) implies that $\mu=0$ as $B$ separates the points of $E$. This is a contradiction to (i). 
Hence, $y'\in l(\mathbb{L}^{\infty})$ and $E'=l(\mathbb{L}^{\infty})$.
\end{proof}

\begin{lemma}\label{contleftinv}
If $l(\LL) = E'$, then the map $l: (\LL, \sigma(\LL,E)) \to (E', \sigma(E',E))$ has a continuous inverse $l^{*}:(E',\sigma(E',E)) \to (\LL,\sigma(\LL,E))$.
\end{lemma}

\begin{proof}
The map $l$ is injective, since its an isometry, and surjective by assumption. 
Thus $l$ has an inverse $l^{*}: E' \to \LL$. 
By the definition of the topology $\sigma(\LL,E)$  sets of the form $U = \{h \in \LL: |B(\mu_{i},h)| <\varepsilon\text{, for all $i=1,\dots,n$}\}$ form a basis for the neighborhoods of $0$ in $\LL$.
Since $l^*$ is linear, it is enough to show that $(l^{*})^{-1}(U)$ is $\sigma(E',E)$-open for such $U$, and since the preimage of the inverse of a bijective function is the image of the original function, we have $(l^{*})^{-1}(U) = l(U)$.
However, if $y\in E'$ is of the form $y=l(h)=l_h$ for some $h\in \LL$, then the bilinear form  $\langle \mu,y \rangle =y(\mu)$ of the dual pair $(E,E')$ becomes
$y(\mu)=l_h(\mu)=B(\mu,h)$.
Since $l(\LL) = E'$, we have
\begin{align*}
  (l^*)^{-1}(U) &= l(U) =  \{l(h) : h\in \mathbb{L}^\infty, |B(\mu_{i},h)| <\varepsilon, i=1,\dots,n\} \\
  & = \{y\in E':  |y(\mu_{i})|<\varepsilon, i=1,\dots,n\}\cap l(\LL)\\
  & = \{y\in E':  |y(\mu_{i})|<\varepsilon, i=1,\dots,n\},
\end{align*}
which is an open set in the $\sigma(E',E)$ topology on $E'$. 
\end{proof}

  \begin{lemma}\label{lemma:0<h<1}
Let $P$ be a probability measure on $(\Omega,\mathcal{F})$, and $h: \Omega\to \mathbb{R}$ measurable. Then $0 \leq h \leq 1$ $P$-a.s. if and only if
\begin{equation}\label{eq:intAh>0}
  0 \leq \int_{\Omega}h dQ \leq 1
\end{equation}
for all probability measures $Q$ on $(\Omega,\mathcal{F})$ with $Q\ll P$.
\end{lemma}
\begin{proof}
Clearly, if $0 \leq h \leq 1$ $P$-a.s., then also $0 \leq h\leq 1$ $Q$-a.s for every $Q\ll P$, and hence
\begin{equation*}
  0\leq \int_{\Omega}h dQ \leq 1.
\end{equation*}
Now, assume that $h$ satisfies \eqref{eq:intAh>0}.
We are only going to show that $h\geq 0$ $P$-a.s, as $h \leq 1$ $P$-a.s. can be shown in the same manner.
Note that $\{h<0\} = \cup_{n\geq 1}A_{n}$, with $A_{n} = \{h \leq -1/n\} \subseteq A_{n+1}$.
If there is an $n$, such that $P(A_{n}) >0$, then the probability measure $Q(F) = P(F\cap A_{n})/P(A_{n})$ satisfies
\begin{equation*}
  \int_{\Omega}h d Q = \frac{1}{P(A_{n})}\int_{A_{n}}h d P \leq  -n^{-1}\frac{P(A_{n})}{P(A_{n})} < 0,
\end{equation*}
which contradicts \eqref{eq:intAh>0}. Hence $P(A_{n}) = 0$ for all $n$, and $P(h<0) = 0$.
\end{proof}

\begin{lemma}\label{lemma:D-s(L,E)-clsd}
$D^{\varepsilon,P}$ and $\widetilde{D}^{\varepsilon \delta, P}$  are $\sigma(\LL,E)$ closed.
\end{lemma}
\begin{proof}
  Note that we can write $D^{\varepsilon,P}$ and $\widetilde{D}^{\varepsilon \delta, P}$ as $D^{\varepsilon,P} = A \cap B$, and $\widetilde{D}^{\varepsilon \delta, P} = A \cap C$ where 
  \begin{equation*}
    \begin{split}
      A &= \{h\in \LL: 0 \leq h \leq 1 \, \mathcal{P}\text{-q.s.}\},\\
      B & = \{h \in \LL: B(P,h) \geq 2 \varepsilon\},\quad \text{and}\\
      C & = \{h \in \LL : B(P,h)  \leq \varepsilon \delta\}.
    \end{split}
  \end{equation*}
  The map $\Phi_{P}: \LL \to \RR, \Phi_{P}(h) = B(P,h)$ is $\sigma(\LL,E)$ continuous, and hence $B = \Phi_{P}^{-1}([2 \varepsilon, \infty))$ and $C = \Phi_{P}^{-1}((-\infty, \varepsilon \delta])$ are $\sigma(\LL, E)$ closed.

    Let us write $\mathcal{M}_{P}$ for the set of probability measures on $(\Omega,\mathcal{F})$ which are absolutely continuous with respect to $P$. 
    Note that the polar set $h^{-1}([0,1]^{c})$ for the property defining $A$ is measurable.
  The measurability of the polar set, and Lemma \ref{lemma:0<h<1} implies that $A$ can be written as
  \begin{equation*}
    \begin{split}
      A &= \{h \in \LL : \forall P \in \mathcal{P}: 0 \leq h \leq 1 \, P\text{-a.s.}\}\\
        & = \bigcap_{P\in \mathcal{P}}\bigcap_{Q\in \mathcal{M}_{P}} \{h \in \LL: 0 \leq  B(Q,h) \leq 1\} \\
        &= \bigcap_{P\in \mathcal{P}}\bigcap_{Q\in \mathcal{M}_{P}} \Phi_{Q}^{-1}([0,1]).
    \end{split}
  \end{equation*}
   Thus $A$ is a $\sigma(\LL,E)$ closed set as an intersection of $\sigma(\LL,E)$ closed sets, and so are $D^{\varepsilon,P}$ and $\widetilde{D}^{\varepsilon \delta, P}$.
\end{proof}

For a subset $D\subset \LL$ define $l(D)=\{l(h): h\in D\}\subset l(\LL)$.

\begin{lemma}\label{weak*compactness_2}
The sets $l(D^{\varepsilon,P})$ and $l(\widetilde{D}^{\varepsilon \delta, P})$ are convex, and if $l(\LL) = E'$ then both sets are $\sigma(E',E)$ compact.
\end{lemma}
\begin{proof}
  Both $D^{\varepsilon,P}$ and $\widetilde{D}^{\varepsilon \delta, P}$ are convex, and since $l$ is linear, so are their images. 
  The set $D^{\varepsilon,P}$ is $\sigma(\LL, E)$ closed by Lemma \ref{lemma:D-s(L,E)-clsd}. 
  Since $l^*$ is continuous by Lemma \ref{contleftinv}, $l(D^{\varepsilon,P})= (l^{*})^{-1}(D^{\varepsilon,P})$ is $\sigma(E',E)$ closed as the continuous preimage of a $\sigma(\LL,E)$ closed set.
  Let $B_{E'} = \{x' \in E': \|x'\|_{E'} \leq 1\}$ denote the closed unit ball in $E'$. 
  $B_{E'}$ is a $\sigma(E',E)$ compact subset of $E'$ by the Banach-Alaoglu Theorem (see \cite{conway:1990:a-course-in-functional-analysis}).
  Since $l$ is an isometry, any $h\in D^{\varepsilon,P}$ satisfies
  \begin{equation*}
    \|l(h)\|_{E'} = \|h\|_{\LL} \leq 1,
  \end{equation*}
  i.e., $l(D^{\varepsilon,P})\subseteq B_{E'}$.
Since $l(D^{\varepsilon,P})$ is $\sigma(E',E)$ closed, this implies that $l(D^{\varepsilon,P})$ is $\sigma(E',E)$ compact.
  By the same argument, $l(\widetilde{D}^{\varepsilon \delta, P})$ is $\sigma(E',E)$ compact.
\end{proof}

\begin{lemma}\label{basicLem}
Let $\mathcal{P}$ and $\mathcal{Q}$ be convex sets of probability measures on $(\Omega,\mathcal{F})$ such that $\mathcal{Q}\lll\mathcal{P}$.
\begin{enumerate}
\item Suppose that the assumptions of Theorem~\ref{HS} are satisfied. Then 
\begin{equation}\label{minmax}
\newinf_{h\in D^{\ep, P}}\sup_{Q\in \mathcal{Q}}E_Q[h]\geq\ep\delta.
\end{equation}
\item  Suppose that the assumptions of Theorem~\ref{dualHS} are satisfied. Then \begin{equation}\label{minmax_dual}
\sup_{h\in \widetilde{D}^{\ep\delta, P}}\newinf_{Q\in \mathcal{Q}}E_Q[h]\leq(2-\ep)\ep.\end{equation}
\end{enumerate}
\end{lemma}

\begin{proof}
Fix $P\in\mathcal{P}$. For the proof of (i), observe that for
$h\in D^{\ep, P}$ we have that

$$2\ep\leq E_P[h]\leq \ep P(h<\ep) + P(h\geq\ep)=(1-\ep)P(h\geq\ep)+\ep.$$

Hence, $P(h\geq\ep)\geq\frac{\ep}{1-\ep}$. Therefore, $A:=\{h\geq\ep\}$ satisfies the condition of Theorem~\ref{HS} as $P(A)\geq\ep$. By assumption, there exists $Q\in\mathcal{Q}$ such that $Q(A)\geq\delta$.
For this $Q$ there exists some $P'\in\mathcal{P}$ such that $Q\ll P'$. As $0\leq h\leq 1$ $\mathcal{P}$-q.s. and therefore $P'(0\leq h\leq 1)=1$ we obtain that $Q(0\leq h\leq1 )=1$.  We conclude that
$$\sup_{Q'\in\mathcal{Q}}E_{Q'}[h]\geq E_Q[h]\geq\ep Q(A)\geq\ep\delta.$$ This can be done for every $h\in D^{\ep, P}$ and so (\ref{minmax}) follows.

For the proof of (ii), observe that for
$h\in \widetilde{D}^{\ep\delta, P}$ we have that
$$\ep\delta\geq E_P[h]> \ep P(h>\ep),$$
hence $P(h>\ep)<\delta$. By assumption, there exists $Q\in\mathcal{Q}$ with $Q(h>\ep)<\ep$. And so, again as $0\leq h\leq 1$ $\mathcal{P}$-q.s. and $Q\ll P'$ for some $P'\in\mathcal{P}$ we get
$$\inf_{Q\in \mathcal{Q}}E_Q[h]\leq E_Q[h]\leq\ep Q(h\leq \ep)+Q(h>\ep)< \ep +(1-\ep)\ep=(2-\ep)\ep.$$
This can be done for every $h\in \widetilde{D}^{\ep\delta,P}$ and so (\ref{minmax_dual}) follows.
\end{proof}

\begin{proof}[Proof of Theorem~\ref{HS}]\label{proof:HS}
In Lemma~\ref{basicLem}  we already saw that the assumptions of Theorem~\ref{HS} imply
\begin{equation}\label{infsup}\newinf_{h\in D^{\ep, P}}\sup_{Q\in \mathcal{Q}}E_Q[h]\geq\ep\delta.\end{equation}
 We can write (\ref{infsup}) as
\begin{equation}\label{infsup2}\newinf_{y\in l(D^{\ep,P})}\sup_{Q\in \mathcal{Q}}\langle Q,y\rangle\geq\ep\delta,\end{equation}
where $\langle \cdot,\cdot\rangle$ is the bilinear form of the dual pair $(E,E')$.
Note that as $y\in l(D^{\ep,P})$ satisfies $y=l_h$ for $h\in D^{\ep,P}$, $\langle Q,y\rangle = l_h(Q)=B(Q,h)=E_Q[h]$ for $y=l_h$.  
By definition of the topology $\sigma(E',E)$, $y\mapsto \langle \mu, y\rangle$ is $\sigma(E',E)$ continuous, for all $\mu\in E$, and as $\mathcal{Q}\lll \mathcal{P}$, obviously $\mathcal{Q}\subseteq E$ hence the continuity holds for all $Q\in\mathcal{Q}$. 
By $\sigma(E',E)$-compactness of the convex set $l(D^{\ep,P})$ (see Lemma~\ref{weak*compactness_2}), Sion's Minimax Theorem applied to  (\ref{infsup2}), which is the same as (\ref{infsup}), now gives
\begin{align}\label{supinf}
\sup_{Q\in \mathcal{Q}}\newinf_{y\in l(D^{\ep,P})}\langle Q,y\rangle = \sup_{Q\in \mathcal{Q}}\newinf_{h\in D^{\ep, P}}E_Q[h] \geq \ep\delta.
\end{align}
As (\ref{supinf})  holds for every fixed $P\in\mathcal{P}$, for each such $P$, by the above inequality, we can choose $Q\in\mathcal{Q}$ such that
$$\inf_{h\in D^{\ep, P}}E_Q[h]\geq\frac{\ep\delta}2.$$
For any $A$ with $P(A)\geq 2\ep$ we have that $h=\ind_A\in D^{\ep, P}$. Therefore

$$Q(A)\geq \inf_{h\in  D^{\ep, P}}E_Q[h]\geq\frac{\ep\delta}{2}.$$
\end{proof}

\begin{proof}[Proof of Theorem~\ref{dualHS}]
  We first note that the version of Sion's Minimax Theorem \ref{thm:sion} cited above does not apply to this setting. 
  However, the assumptions of this theorem are more than satisfied. 
  The function $f(Q,h) = \mathbb{E}_{Q}( h )$ is bi-linear, and hence convex and concave in both its argument.
  Additionally, $f$ is continuous in both its arguments. 
    Hence, the function $-f(x,y)$ satisfies the conditions of Theorem \ref{thm:sion}.
  Now, recall that for an arbitrary family $W \subseteq \mathbb{R}$, we have
  \begin{equation*}
    \inf_{w\in W} w = - \sup_{w\in W} (-w).
  \end{equation*}
  Combining this idea with Theorem \ref{thm:sion}, and the fact that $l(\widetilde{D}^{\varepsilon,P})$ is $\sigma(E',E)$ compact (see Lemma \ref{weak*compactness_2}) yields
  \begin{equation*}
    \begin{split}
    \inf_{Q\in \mathcal{Q}} \sup_{h\in \widetilde{D}^{\varepsilon,P}} \mathbb{E}_{Q}( h ) & = - \sup_{Q\in \mathcal{Q}}\inf_{h\in l(\widetilde{D}^{\varepsilon,P})} - \mathbb{E}_{Q}( h )\\
                                                                                          &= - \inf_{h\in l(\widetilde{D}^{\varepsilon,P})}\sup_{Q\in \mathcal{Q}} - \mathbb{E}_{Q}( h )\\
                                                                                          &=  \sup_{h\in \widetilde{D}^{\varepsilon,P}} \inf_{Q\in \mathcal{Q}} \mathbb{E}_{Q}( h ).
    \end{split}
  \end{equation*}

By  an analogous application of Lemma~\ref{basicLem} as in the proof of Theorem~\ref{HS}, we get
$$ \inf_{Q\in \mathcal{Q}}\sup_{h\in  \widetilde{D}^{\ep\delta, P}}E_Q[h]\leq(2-\ep)\ep.$$
Now, choose a $Q\in\mathcal{Q}$ such that $\sup_{h\in  \widetilde{D}^{\ep\delta, P}}E_Q[h]<2\ep$.
For any $A$ with $P(A)<\ep\delta$ we have that $h=\ind_A\in\widetilde{D}^{\ep\delta, P}$, and hence
$$Q(A)\leq \sup_{h\in  \widetilde{D}^{\ep\delta, P}}E_Q[h]<2\ep.$$
\end{proof}

\section{Robust large financial markets in discrete time}
\label{sec:RLFM}

A robust large financial market, i.e., a large financial market under model uncertainty, is here constructed similarly to the case without uncertainty through a \emph{sequence} of robust financial markets (which we will call small robust markets).
For these small robust  markets, we choose a discrete time one-period setting as in \cite{BN:15}.
As we have to assume that $l(\LL)=E'$ in the quantitative Halmos-Savage Theorems we need to ensure that this condition is satisfied in the small robust markets. 
A large class of models admits an extension on which $l(\LL) = E'$, see \cite{klein-koestenberger-robust-duality}.
In particular, this condition is satisfied in the robust binomial model of \cite{blanch_carassus} for the  one-period case, and in a Black-Scholes model with uncertain constant volatility (bounded away from $0$) and uncertain constant drift.
For an extensive analysis of the large financial market based on robust one-period binomial models, we refer to Section \ref{App_robust_bin}.

A robust large financial market can also be seen as a model of a robust financial market including infinitely many assets, by choosing the dimension $d(n)$ of the small robust markets such that $d(n)\to\infty$, for $n\to\infty$. 
Let us now give the details of our model of a large financial market.
For each $n\ge1$, the robust small market $n$ is defined as follows. Let $(\Omega^n,\mathcal{F}^n)$ be a  measure space where the filtration is given by $\mathcal{F}^n_0=\{\emptyset,\Omega\}$ and $\mathcal{F}^n_1=\mathcal{F}^n$. Let 
$\mathcal{P}^n$ be a convex set of probability measures on $(\Omega^n,\mathcal{F}^n)$ such that Assumption~\ref{closed} below holds. 
\begin{assumption}\label{closed}
$\mathcal{P}^n$ satisfies $l(\LL(\mathcal{P}^n))=E_n'$ for all $n\geq1$, where $E_n$ and $E_n'$ for $\mathcal{P}^n$ be defined as in Section~\ref{secHS}.  
\end{assumption}
The trading objects are $d(n)$ stocks 
$S^n_t=(S^{n,1}_t,\dots, S^{n,d(n)}_t)$, $t=0,1$. Here $S^n_0$ is a deterministic vector in $\mathbb{R}^{d(n)}$ and $S^n_1$ is an $\mathcal{F}^n$-measurable $\mathbb{R}^{d(n)}$-valued random vector. 
We think of $S_t^n$ as the (discounted) prices of the $d(n)$ stocks at time $t$, $t=0,1$.

In this setting a trading strategy $H^n=(H^{n,1},\dots, H^{n,d(n)})$ is a deterministic vector in $\mathbb{R}^{d(n)}$. A portfolio with strategy $H^n$ is given by
\begin{equation}\label{pf_market_n}
X^n_1=H^n(S^n_1-S^n_{0}),
\end{equation}
where $X^n_0=0$ and $H^n(S^n_1-S^n_{0}):=\sum_{k=1}^{d(n)}H^{n,k}(S^{n,k}_1-S^{n,k}_{0})$.
We now give the definition of robust no-arbitrage on the  market $n$ as in \cite{BN:15}:

\begin{definition}\label{na_market_n}
The market $n$  satisfies the condition NA$(\mathcal{P}^n)$ if for all $H^n\in\mathbb{R}^{d(n)}$
$$H^n(S^n_1-S^n_{0})\geq 0\quad\text{$\mathcal{P}^n$-q.s.}\quad\text{implies}\quad H^n(S^n_1-S^n_{0})= 0\quad\text{$\mathcal{P}^n$-q.s.}$$
\end{definition}

Now, we can give the definition of a robust large financial market in discrete time.

\begin{definition}\label{robust_lfm}
A robust large financial market is a sequence of markets $n$ as given above, i.e., a sequence of $(\Omega^n,\mathcal{F}^n,(\mathcal{F}_t^n)_{t=0, 1})$ with a $d(n)$-dimensional stock in one time period $S^n_t=(S^{n,1}_t,\dots, S^{n,d(n)}_t)$, $t=0,1$, where $d(n)$ is the number of stocks in market $n$.
\end{definition}

We will assume that each market $n$ satisfies the robust no-arbitrage condition of Definition~\ref{na_market_n}.

\begin{assumption}\label{na_ass}
For each $n\geq1$, we have that  NA$(\mathcal{P}^n)$ holds.
\end{assumption}

Let us recall the connection to martingale measures in the one-period model of \cite{BN:15}. On market $n$ we define the following set $\mathcal{Q}^n$ of probability measures on $(\Omega^n,\mathcal{F}^n)$. For a random variable $f:\Omega\to[0,\infty]$ and any probability measure $P$ the expected value is defined as 
\begin{equation}\label{convention}
\text E_P[f]=E_P[f^+]-E_P[f^-],
\end{equation}
with the convention $\infty-\infty=-\infty$. We recall from Definition \ref{lll}, that $Q\lll\mathcal{P}$ means 
that there exists $P \in \mathcal{P}$ such that $Q\ll P$.

\begin{definition}\label{martmeas_market_n}
A probability measure $Q^n$ on $(\Omega^n,\mathcal{F}^n)$ is called martingale measure (on market $n$), if
$E_{Q^n}[S^{n,i}_1-S^{n,i}_{0}]=0$, for all $i=1,\dots, d(n)$, using Convention~\eqref{convention}.
Define
$$\mathcal{Q}^n=\{Q^n\lll\mathcal{P}^n: Q^n \text{ martingale measure}\}.$$
\end{definition}

A consequence of Assumption~\ref{na_ass} is that on each small market $n$, the property of the set $\mathcal{Q}^n$ as in (ii) of the following first FTAP holds. Moreover, the superhedging price is given in the Superhedging Theorem below, see \cite{BN:15} for both results.

\begin{theorem}
\label{ftap_n}
The following are equivalent:
\begin{enumerate}
\item NA$(\mathcal{P}^n)$ holds.
\item For all $P^n\in\mathcal{P}^n$ there exists $Q^n\in \mathcal{Q}^n$ such that $P^n\ll Q^n$, i.e., $\mathcal{P}^n\lll \mathcal{Q}^n$.
\end{enumerate}
\end{theorem}

\begin{theorem}
\label{super_n}
Let NA$(\mathcal{P}^n)$ hold and let $f$ be a random variable. Then
\begin{align*}\sup_{Q^n\in\mathcal{Q}^n}E_{Q^n}[f] =\pi(f),\end{align*}
where $\pi(f):=\inf\{x\in\mathbb{R}: \exists H^n\in\mathbb{R}^{d(n)}\text{ with }x+H^n(S^n_1-S^n_0)\geq f\text{ $\mathcal{P}^n$-q.s.}\}$.
Moreover, if $\pi(f)>-\infty$, then there exists $H^n\in\mathbb{R}^{d(n)}$ such that
$$\pi(f)+H^n(S^n_1-S^n_0)\geq f,\text{ $\mathcal{P}^n$-q.s.}$$

\end{theorem}

\begin{remark}
  In the one-period setting, we can apply the theory from \cite{klein-koestenberger-robust-duality} to guarantee that $l(\LL) = E'$ in the example of the robust binomial model of  \cite{blanch_carassus} for $T=1$.
  This is no longer the case in the robust multi-period binomial model ($T>1$). It is  an open question, whether any (non-trivial) robust multi-period binomial model satisfies $l(\LL) = E'$. 
   Given the assumption that $l(\LL) = E'$, it is conceptually straight-forward, but somewhat lengthy and technical to extend  Theorem~\ref{lfmftap} and Theorem~\ref{lfmftap2} to the multi-period case.   
   However, currently it is not clear that the extension of our results to the multi-period setting is useful,
  as it is unclear if the assumption $l(\LL) = E'$ holds in the robust multi-period binomial model.
  Hence, in the present paper, we restrict our attention to the one-period case. 
\end{remark}

We  now introduce  no asymptotic arbitrage conditions on the large financial market, that is, for a sequence of  small  arbitrage-free markets. We  first consider a variant of {\em no asymptotic arbitrage of first kind} (NAA1), originally introduced in \cite{KK:94}. 
This condition is closely related to the  notion of  {\em no unbounded profit with bounded risk} (NUPBR). 
In \cite{CKT:14} the equivalence of these two conditions for a large financial market in continuous time on a fixed probability space $\Omega$ has been established.
In addition, we will also use a variant of {\em no asymptotic arbitrage of second kind} (NAA2), which also goes back to \cite{KK:94}.

\begin{definition}\label{AA1}
We say that the robust large financial market has an asymptotic arbitrage of first kind (AA1$(\mathcal{P}^n$)) if the following holds: there exists a subsequence of markets $n_k$ and a sequence of portfolios $X^k_1=H^{n_k}(S^{n_k}_1-S^{n_k}_0)$ and a sequence of positive real numbers $c_k\to0$ such that
\begin{enumerate}
\item for all $k\geq 1$,  $X^k_1\geq -c_k$ $\mathcal{P}^{n_k}$-q.s., and
\item  there exists $\alpha>0$ and a sequence $(P^k)_{k\geq1}$  with $P^k\in\mathcal{P}^{n_k}$ such that $$P^k(X^k_{1}\geq \alpha)\geq\alpha,$$ for all $k\geq1$.
\end{enumerate}
We say that {\em no asymptotic arbitrage of the first kind}, NAA1$(\mathcal{P}^n)$, is satisfied if the above does not exist. 
\end{definition}
Note that by modifying the trading strategy, the profit in (ii) above can become unboundedly high in the limit, while the risk still stays vanishing. Indeed, choose $\tilde{H}^{n_k}=\frac{1}{\sqrt{c_k}}H^{n_k}$.

\begin{definition}\label{AA2}
We say that the robust large financial market has an asymptotic arbitrage of the second kind (AA2($\mathcal{P}^n$)) if the following holds: there exists a subsequence of markets $n_k$ and a sequence of portfolios $X^k_1=H^{n_k}(S^{n_k}_1-S^{n_k}_0)$ such that
\begin{enumerate}
\item for all $k\geq 1$,  $X^k_1\geq -1$ $\mathcal{P}^{n_k}$-q.s., and
\item there exists $\alpha>0$ and a sequence $(P^k)_{k\geq1}$  with $P^k\in\mathcal{P}^{n_k}$ such that $$\lim_{k\to\infty}P^k(X^k_{1}\geq \alpha)=1,$$ for all $k\geq1$.
\end{enumerate}
We say that {\em no asymptotic arbitrage of second kind}, NAA2$(\mathcal{P}^n)$, is satisfied if the above does not exist.
\end{definition}

We will now formulate two robust versions of the Fundamental Theorem of Asset Pricing (FTAP) for large financial markets in discrete time. First, we will present a FTAP with the notion NAA1($\mathcal{P}^n)$.  For the convenience of the reader, we give the definition of contiguity for sequences of probability measures which, as usual in the classical setting of large financial markets on a sequence of $(\Omega^n)_{n\geq1}$, replaces the notion of absolute continuity of probability measures. 

\begin{definition}\label{contiguity}
For $n\geq1$ let $P^n$ and $Q^n$ be probability measures on $(\Omega^n,\mathcal{F}^n)$. Then $(P^n)_{n\geq 1}$ is contiguous with respect to $(Q^n)_{n\geq1}$ if for each sequence $A^n\in\mathcal{F}^n$ with $Q^n(A^n)\to0$ for $n\to\infty$ we have that $P^n(A^n)\to0$ for $n\to\infty$. In this case we use the notation $(P^n)_{n\geq1}\triangleleft (Q^n)_{n\geq1}$.
\end{definition}

\begin{theorem}[Robust FTAP for large financial markets with NAA1]\label{lfmftap}
Let a large financial market as above be given such that Assumptions~\ref{closed} and \ref{na_ass} hold. Then NAA1$(\mathcal{P}^n)$ $\Leftrightarrow$ for each sequence $(P^n)_{n\geq1}$ with $P^n\in \mathcal{P}^n$, for all $n$, there exists a sequence
 $(Q^n)_{n\geq1}$ with $Q^n\in \mathcal{Q}^n$, for all $n$, such that $(P^n)_{n\geq1}\triangleleft (Q^n)_{n\geq1}$.
\end{theorem}

Moreover, for NAA2($\mathcal{P}^n$), we will need a weak version of contiguity that appeared first in \cite{KS:96} in the analogous context and was denoted as \emph{weak contiguity} in \cite{KK:98}. 
We cannot get a stronger property for the sets of equivalent martingale measures, as this already fails in the classical case. A counterexample can be found in \cite{KS:96}.

\begin{definition}\label{weak_contiguity}
For each $n\geq1$ let  $P^n$  be a probability measure on $(\Omega^n,\mathcal{F}^n)$. Moreover, let $\mathcal{R}^n$ be a convex set of probability measures on $(\Omega^n,\mathcal{F}^n)$.
Then the sequence of sets $(\mathcal{R}^n)_{n\geq 1}$ is {\emph  weakly contiguous with respect to} $(P^n)_{n\geq1}$ if for every $\ep>0$ there exists $\delta>0$ and a sequence $(Q^{n,\ep})_{n\geq 1}$ with $Q^{n,\ep}\in\mathcal{R}^n$  such that, for all $n\geq1$ and $A^n\in\mathcal{F}^n$ with $P^n(A^n)<\delta$, we have that $Q^{n,\ep}(A^n)<\ep$. 
In this case, we use the notation $(\mathcal{R}^n)_{n\geq1}\triangleleft_w (P^n)_{n\geq1}$.
\end{definition}

\begin{theorem}[Robust FTAP for large financial markets with  NAA2]\label{lfmftap2}
Let a large financial market as above be given such that the Assumptions~\ref{closed} and \ref{na_ass} hold. Then NAA2$(\mathcal{P}^n)$ $\Leftrightarrow$ for each sequence $(P^n)_{n\geq1}$ with $P^n\in \mathcal{P}^n$, for all $n$, we have that
 $(\mathcal{Q}^n)_{n\geq1}\triangleleft_w(P^n)$.
\end{theorem}

\subsection{Proofs of Theorem~\ref{lfmftap} and Theorem~\ref{lfmftap2}}
First, we will prove two useful lemmas that relate the absence of asymptotic arbitrage to properties of the sets of martingale measures. With the help of these technical lemmas, we can apply the two quantitative Halmos-Savage Theorems of Section~\ref{secHS} to prove the two FTAPs.

\begin{lemma}\label{tech_lemma} Assume that  NAA1$(\mathcal{P}^n)$ holds. Then for each $\ep>0$ there exists $\delta>0$ such that, for all $n\geq 1$ and all $A^n\in\mathcal{F}^n$ with the property that there exists $P^n\in\mathcal{P}^n$ with $P^n(A^n)\geq\ep$, there exists $Q^n\in\mathcal{Q}^n$ with $Q^n(A^n)\geq\delta$.
\end{lemma}

\begin{proof}
Suppose NAA1 holds and the $\ep$-$\delta$ condition would not hold. That means there exists $\ep>0$ such that for all $\delta>0$ there exists an $n=n(\delta)$ and $A^n\in\mathcal{F}^n$ such that there exists $P^n\in\mathcal{P}^n$ with $P^n(A^n)\geq\ep$ but
$$\sup_{Q^n\in\mathcal{Q}^n}Q^n(A^n)<\delta.$$ Obviously $f=\ind_{A^n}$ is measurable and $E_Q[f]=Q(A^n)\geq 0$, for all $Q^n$.
We can thus apply Theorem~\ref{super_n} to get, for $f=\ind_{A^n}$,
$$\pi(f)=\sup_{Q^n\in\mathcal{Q}^n}Q^n(A^n)<\delta.$$
Moreover $\pi(f)\geq0>-\infty$, and so there exists $H^n$ such that
\begin{equation}\label{asarb}
\delta+H^n(S^n_{1}-S^n_0)\geq f=\ind_{A^n}\quad\mathcal{P}^n\text{-q.s.}
\end{equation}
Inequality (\ref{asarb}) above allows to construct an AA1. Let $X^n_1=H^n(S^n_{1}-S^n_0)$. Then by (\ref{asarb})  $X^n_1\geq -\delta$, $\mathcal{P}^n$-q.s.  Moreover,
$$P^n(X^n_1\geq1-\delta)\geq P^n(A^n)\geq\ep.$$
This holds for fixed $\ep$ and every $\delta$. So, for $\delta_k\downarrow0$ choose $n_k=n(\delta_k)$ and $\alpha=\min(\ep,1-\delta_1)$ and $P^k\in\mathcal{P}^{n_k}$ as above. This gives an AA1.
\end{proof}

\begin{lemma}\label{tech_lemma2} Assume that NAA2$(\mathcal{P}^n)$ holds. Then for each $\ep>0$ there exists $\delta>0$ such that, for all $n\geq 1$ and all $A^n\in\mathcal{F}^n$ with the property that there exists $P^n\in\mathcal{P}^n$ with $P^n(A^n)<\delta$, there exists $Q^n\in\mathcal{Q}^n$ with $Q^n(A^n)<\ep$.
\end{lemma}

\begin{proof}
Suppose NAA2 holds and the $\ep$-$\delta$ condition would not hold. That means there exists $\ep>0$ such that for all $\delta>0$ there exists an $n(\delta)$   and $A^n\in\mathcal{F}^n$ such that there exists $P^n\in\mathcal{P}^n$ with $P^n(A^n)<\delta$ but
$\inf_{Q^n\in\mathcal{Q}^n}Q^n(A^n)\geq\ep$. This obviously implies for $\overline{A^n}=\Omega^n\setminus A^n$ that $P^n(\overline{A^n})\geq 1-\delta$ and
$$\sup_{Q^n\in\mathcal{Q}^n}Q^n(\overline{A^n})<1-\ep.$$
Apply again Theorem~\ref{super_n}, now for $f=\ind_{\overline{A^n}}$, to find
$H^n$ such that
\begin{equation}\label{asarb2}
1-\ep+H^n(S^n_1-S^n_0)_{1}\geq \ind_{\overline{A^n}}\quad\mathcal{P}^n\text{-q.s.}
\end{equation}
Inequality (\ref{asarb2}) above allows to construct an AA2. Let $X^n_1=H^n(S^n_{1}-S^n_0)$. As in the proof of Lemma~\ref{tech_lemma}  it follows that  $X^n_1\geq -1+\ep$ $\mathcal{P}^n$-q.s.  Moreover, by (\ref{asarb2}), 
$$P^n(X^n_1\geq\ep)\geq P^n(\overline{A^n})\geq1-\delta.$$
This holds for fixed $\ep$ and every $\delta$. So, for $\delta_k\downarrow0$ choose $n_k=n(\delta_k)$ and $\alpha=\ep$ and $P^k\in\mathcal{P}^{n_k}$ as above. This gives an AA2.
\end{proof}

In order to make use of the quantitative Halmos-Savage type theorems from Section~\ref{secHS}, we need convexity of the sets $\mathcal{Q}^n$. We give the easy proof.

\begin{lemma}\label{setQconvex}
The set $\mathcal{Q}^n$ is convex, for all $n\geq1$.
\end{lemma}

\begin{proof}
By assumption $\mathcal{P}^n$ is convex. Let $Q_1$, $Q_2\in\mathcal{Q}$ (where we skip the index $n$ for ease of notation). Then $Q_1\ll P_1$ and $Q_2\ll P_2$ for some measures $P_1,P_2\in\mathcal{P}$. As $\mathcal{P}$ is convex $P=\frac{P_1+P_2}{2}\in \mathcal{P}$. Clearly each convex combination $Q=\lambda Q_1+(1-\lambda) Q_2$ for $0<\lambda<1$ satisfies: $Q\ll P$.
The martingale property  for $Q$ follows easily  as  $Q_1$ and $Q_2$ are martingale measures and $E_Q[\,\cdot\,]=\lambda E_{Q_1}[\,\cdot\,]+(1-\lambda)E_{Q_2}[\,\cdot\,]$.
\end{proof}

\begin{proof}[Proof of Theorem~\ref{lfmftap}]
Assume first that NAA1 holds. Lemma~\ref{tech_lemma} above gives that for each $m\geq1$ and $\ep_m=\frac1{m}$ there exists $\delta_m$ such that the condition of the lemma holds with $\ep_m$ and $\delta_m$. Fix a sequence $(P^n)_{n\geq1}$ with $P^n\in\mathcal{P}^n$. 
Lemma~\ref{tech_lemma} implies that the assumption of Theorem~\ref{HS} is satisfied for $\ep_m$ and $\delta_m$. 
Use Theorem~\ref{HS} to find, for $\ep_m$, $\delta_m$, and $n\geq1$, the measure $Q^{n,m}\in\mathcal{Q}^n$ such that for all $n$ and $A^n\in\mathcal{F}^n$ with $P^n(A^n)\geq 2\ep_m$ we have $Q^{n,m}(A^n)\geq\frac{\ep_m\delta_m}{2}$. Define a new measure $Q^n$ as follows:
\begin{equation}\label{q}
Q^n=\frac1{1-2^{-n}}\sum_{m=1}^n2^{-m}Q^{n,m}.
\end{equation}
By  Lemma~\ref{setQconvex} $\mathcal{Q}^n$ is convex. Hence, we have that $Q^n\in\mathcal{Q}^n$. We claim that $(P^n)_{n\geq1}\triangleleft (Q^n)_{n\geq1}$. The proof of this claim works exactly as in \cite{KS:96}. For the convenience of the readers, we recall the proof here: let $A^n\in\mathcal{F}^n$ with $Q^n(A^n)\to0$. Suppose $P^n(A^n)$ does not converge to $0$ for $n\to\infty$. Then there exists $\alpha>0$ and a subsequence (denoted by $n'$) such that
$$P^{n'}(A^{n'})\geq\alpha,$$
for all $n'$.  Choose $N_0\geq 1$ such that $\alpha>\frac2{N_0}=2\ep_{N_0}$. Define $\beta=2^{-N_0}\frac{\ep_{N_0}\delta_{N_0}}{2}>0$. Hence
$$P^{n'}(A^{n'})\geq2\ep_{N_0},$$
for all $n'$. By the above, we have that $Q^{n',N_0}(A^{n'})\geq \frac{\ep_{N_0}\delta_{N_0}}{2}$, for all $n'$. Therefore we have that, for all $n'\geq N_0$,
\begin{align*}
    Q^{n'}(A^{n'})
    &=\frac1{1-2^{-n'}}\sum_{m=1}^{n'}2^{-m}Q^{n',m}(A^{n'}) \\
    &\geq\frac1{1-2^{-n'}}2^{-N_0}Q^{n',N_0}(A^{n'})\geq\frac1{1-2^{-n'}}\beta,
\end{align*}
a contradiction to $Q^n(A^n)\to0$, for $n\to\infty$.

For the other direction, assume that for each sequence $(P^n)_{n\geq1}$ with $P^n\in\mathcal{P}^n$ there exists a sequence $(Q^n)_{n\geq 1}$ with $Q^n\in\mathcal{Q}^n$ such that $(P^n)_{n\geq 1}\triangleleft (Q^n)_{n\geq 1}$. 
We will show NAA1. 
Indeed, assume that an AA1 exists. 
Then there exists a subsequence $n_k$ and integrands $H^{n_k}$ such that, for all $k\geq1$, $X^k_{1}=H^{n_k}(S^{n_k}_{1}-S^{n_k}_{0})\geq -c_k$ $\mathcal{P}^{n_k}$-q.s.
Moreover, there exists a sequence $P^k\in\mathcal{P}^{n_k}$ such that $P^k(X^k_{1}\geq\alpha)\geq\alpha$, for all $k\geq1$,  for some $\alpha>0$.
Set $A^k=\{X^k_{1}\geq\alpha\}$, and define the sequence $(P^n)_{n\geq1}$ as follows: for $n=n_k$, $P^n=P^{n_k}$, and for $n\notin\{n_k,k\geq1\}$ choose an arbitrary $P^n\in\mathcal{P}^n$. 
Then there exists $(Q^n)_{n\geq 1}$ such that $(P^n)_{n\geq 1}\triangleleft (Q^n)_{n\geq 1}$. 
In particular, this property holds for the subsequences $n_k$, $k\geq1$. 
As $P^{n_k}(A^k)\geq\alpha$ for all $k$ there exists $\beta>0$ such that $Q^{n_k}(A^k)\geq\beta$ for all $k$ (otherwise this leads to a contradiction to $(P^n)_{n\geq 1}\triangleleft (Q^n)_{n\geq 1}$). 
For each $Q^k$ there exists $\tilde{P}^k\in\mathcal{P}^{n_k}$ with $Q^k\ll \tilde{P}^k$. 
Hence, as $X^k_{1}\geq -c_k$, $\mathcal{P}^{n_k}$-q.s. (and therefore  $\tilde{P}^k$-a.s.) for all $k$, we get $Q^k(X^k_{1}\geq -c_k)=1$, and 
$$E_{Q^k}[X^k_{1}]\geq -c_kQ^k(\Omega^{n_k}\setminus A^k)+\alpha Q^k(A^k)>-c_k+\alpha\beta>0,$$
for $k$ large enough. 
This is a contradiction to the fact that $Q^k$ is a martingale measure for $S^{n_k}$. Hence, there cannot exist an AA1.

\end{proof}

Finally, we will give the proof of Theorem~{\ref{lfmftap2}}, which immediately follows from Lemma~\ref{tech_lemma2}. 

\begin{proof}[Proof of Theorem~\ref{lfmftap2}]
Assume first that NAA2 holds.  Lemma~\ref{tech_lemma2} above  implies that for each $\ep>0$ there exists $\delta=\delta(\ep)>0$ such that for all $n$ the assumptions of Theorem~\ref{dualHS} are satisfied (for this $\ep$ and $\delta(\ep)$). 
Fix now a sequence $(P^n)_{n\geq1}$ with $P^n\in\mathcal{P}^n$, and fix an arbitrary $\widetilde{\ep}>0$. 
Choose $\ep$ such that $2\ep<\widetilde{\ep}$, set $\delta=\delta(\ep)$, and define $\widetilde{\delta}=\ep\delta$. 
By Theorem~\ref{dualHS} there exists $Q^{n, \widetilde{\ep}}\in\mathcal{Q}^n$ such that for all $A^n\in\mathcal{F}^n$ with $P^n(A^n)<\widetilde{\delta}=\ep\delta$ we have that $Q^{n, \widetilde{\ep}}<\widetilde{\ep}=2\ep$. Hence  $(\mathcal{Q}^n)_{n\geq1}\triangleleft_w (P^n)_{n\geq 1}$.

For the other direction, assume that for every sequence $(P^n)_{n\geq1}$, $P^n\in\mathcal{P}^n$, we have that $(\mathcal{Q}^n)_{n\geq1}\triangleleft_w(P^n)_{n\geq 1}$.
We will show NAA2. Indeed, suppose an AA2 exists. 
Then there exists a subsequence $n_k$ and integrands $H^{n_k}$ such that, for all $k\geq1$, $X^k_{1}=H^{n_k}(S^{n_k}_{1}-S^{n_k})_{0}\geq -1$ $\mathcal{P}^{n_k}$-q.s. 
Moreover, there exists a sequence $P^k\in\mathcal{P}^{n_k}$ such that $P^k(X^k_{1}\geq\alpha)\to1$, for $k\to\infty$,  and some $\alpha>0$. 
Fix a sequence $(P^n)_{n\geq 1}$ as follows: for $n=n_k$, $P^n=P^{n_k}$, and for $n\notin\{n_k,k\geq1\}$ choose an arbitrary $P^n\in\mathcal{P}^n$. Fix 
 $\ep>0$  small enough such that $-\ep+\alpha(1-\ep)>0$. 
 By $(\mathcal{Q}^n)_{n\geq1}\triangleleft_w(P^n)_{n\geq 1}$ for this $\ep$ there exists $\delta>0$ and a sequence $(Q^{n,\ep})_{n\geq1}$ as in Definition~\ref{weak_contiguity}.
Define $A^k=\{X^k_{1}<\alpha\}$, then for $k$ large enough $P^{n_k}(A^k)<\delta$. This implies that $Q^{n_k,\ep}(A^k)<\ep$, for all $k\geq1$.
For this $Q^{n_k,\ep}$ there exists $\tilde{P}^k\in\mathcal{P}^{n_k}$ with $Q^{n_k,\ep}\ll \tilde{P}^k$. 
Hence, $Q^{n_k,\ep}(X^k_{1}\geq -1)=1$, because $X^k_{1}\geq-1$, $\mathcal{P}^{n_k}$-q.s. 
So, we get
$$E_{Q^{n_k,\ep}}[X^k_{1}]\geq -Q^{n_k,\ep}(A^k)+\alpha Q^{n_k,\ep}(\Omega^{n_k}\setminus A^k)\geq-\ep+\alpha(1-\ep)>0,$$
for $k$ large enough. This is a contradiction to the fact that $Q^{n_k,\ep}$ is a martingale measure for $S^{n_k}$. Hence, there cannot exist an AA2.

\end{proof}

\section{An application to the one-period robust binomial model}\label{App_robust_bin}

In this section, we consider a sequence of one-period robust binomial models following Section~5.3 and the appendix of \cite{klein-koestenberger-robust-duality}. For each $n$, this is exactly the robust binomial model of  \cite{blanch_carassus} for $T=1$. We shall also use  notation and results of \cite{liebrich2022}. In this section the dimension is $d(n)=1$ throughout.
Let $\Omega \equiv \Omega^n\equiv(0,+\infty)$, equipped with the Borel $\sigma$-algebra $\mathcal{F}=\mathcal{B}(\Omega)$. 
The (discounted) price process of the risky asset 
is given by $S^n_0=1$ and 
$$S^n_{1}=S^n_0\cdot Y^n_{1},$$
where $Y^n_{1}: (0,+\infty)\to(0,+\infty)$ is a bijective measurable map, for example, $Y^n_1(\omega)=\omega$. Let $\mathcal{M}_1$ be the set of all probability measures on $(\Omega,\mathcal{F})$. Moreover, we consider lower and upper bounds for 
the probabilities as well as for the down- and up-moves of the binomial model $n =1,2,\dots$,
  such that in particular the assumptions of \cite{blanch_carassus} are satisfied for the one-period model. We will make some additional assumptions and comment on these in Remark~\ref{Remark_Comments_Ass} below. More precisely, we shall consider the following conditions on the model parameters, where the intervals $[d_n, D_n], [u_n, U_n]$ stand for the ranges of the down- and up-factors of the $n$th model, while $[\pi_n,\Pi_n]$ and $[1-\Pi_n,1-\pi_n]$ denote the range of  the respective probabilities.

\begin{mass}\label{assumptions}
Let $u^n, U^n, d^n, D^n,\pi^n, \Pi^n\in\mathbb{R}$. We make the following set of assumptions for all $n\in \mathbb{N}$:
\begin{enumerate}[label=\hspace{1\parindent}(A\arabic*),ref=(A\arabic*),leftmargin=*,align=left]
  \item \label{ass_blan_car_1step} The model parameters satisfy:
\begin{enumerate}
  \item[(i)] $0<{\pi}^n\leq  {\Pi}^n<1$,
  \item[(ii)] $d^n\leq D^n$, $u^n\leq U^n$,
  \item[(iii)] $0<d^n<1<U^n$.
\end{enumerate}
\item \label{ass34} $d^n<\min(u^n,D^n)$ and $\max(u^n,D^n)<U^n$. 
\item \label{par_limits} There exist $d\in[0,1]$, $D\in[0,+\infty]$, $u\in[0,+\infty]$, $U\in[1,+\infty]$ such that the sequences of parameters converge: $\lim_{n\to\infty}d^n=d$, $\lim_{n\to\infty}D^n=D$, $\lim_{n\to\infty}u^n=u$, $\lim_{n\to\infty}U^n=U$. 
\item \label{avoid_riskless_limit} The limit points of the model parameters satisfy:
  \begin{enumerate}
    \item[(i)] $d<U$, 
    \item[(ii)] $[d,D]\neq\{1\}$ or $\limsup\Pi^n>0$,
    \item[(iii)]  $[u,U]\neq\{1\}$ or $\liminf\pi^n<1$.
  \end{enumerate}

\end{enumerate}
\end{mass}

Let $I^n\subset \mathbb{R}^3$ be defined as $I^n=[u^n,U^n]\times[d^n,D^n]\times [{\pi}^n,{\Pi}^n]$ and let $\mathcal{L}^n_1=\{\pi\delta_{u}+(1-\pi)\delta_{d}: (u,d,\pi)\in I^n\}$, where $\delta_x$ denotes the Dirac measure with $\delta_x(A)=1$, for $A\in\mathcal{F}$ iff $x\in A$ and $\delta_x(A)=0$ otherwise.
For a probability measure $R$ on $(\Omega, \mathcal F)$ we denote by $R\circ (Y_1^n)^{-1}$ the law of $Y_1^n$ with respect to $R$, that is, for $A\in\mathcal{F}$, $R\circ (Y_1^n)^{-1}(A):=R(Y_1^n\in A)$. 

\begin{remark}\label{Remark_Comments_Ass}
Let us comment on the set of conditions introduced in Assumption \ref{assumptions}:
\begin{enumerate}
 \item Assumption \ref{ass_blan_car_1step} corresponds to the assumptions  of \cite{blanch_carassus}  for the one-period model.
\item 
  Assumption \ref{ass34}  corresponds to  Assumption 39 in
\cite{klein-koestenberger-robust-duality}, allowing us to apply their results.
It has the following consequences for each market $n$: first, there is real uncertainty in the up value as well as in the down value  since $d^n <D^n$ and $u^n <U^n$. Second, the smallest possible down value is strictly less than the smallest possible up value, and the largest possible down value is strictly less than the largest possible up value. 
\item 
Assumption~\ref{par_limits} does not restrict the possible values of limits of $D^n$ and $u^n$. Moreover, note that $0\leq d\leq 1$ and $1\leq U\leq+\infty$ for each accumulation point of the corresponding sequence. So, Assumption~\ref{par_limits} does not restrict the possible limits of $U^n$, $d^n$ as well.  More generally, one could assume Assumption \ref{par_limits} for all possible choices of accumulation points and still would get the following results. Note also that we do not assume anything on the convergence of $(\pi^n)$ or $(\Pi^n)$.

 Observe that as a special case of our setting we may recover the classical binomial model by letting $d^n=D^n$, $u^n=U^n$, and $\pi^n=\Pi^n$. Assumption~\ref{par_limits} then implies that the down- and up-states of the binomial model converge to $u$ and $d$, respectively, while no assumptions are made on the convergences of the associated probabilities $(\pi^n)$.
For details on possible limit models see Subsection~\ref{limit model}. 
\item Assumption \ref{avoid_riskless_limit} guarantees that the large financial market  does not converge to the riskless asset with constant value $1$ (in discounted terms). 
Indeed, we speak of convergence to the riskless asset if one of the following conditions holds:
\begin{enumerate}\item[(a)] $[d^n,D^n]\times[u^n, U^n]\to \{1\}\times\{1\}$, for $n\to\infty$,
    \item[(b)] $[d^n,D^n]\to\{1\}$ and $[\pi^n,\Pi^n]\to\{0\}$,
     \item[(c)] $[u^n,U^n]\to\{1\}$ and $[\pi^n,\Pi^n]\to\{1\}$.\end{enumerate}
The interpretation is clear: in (a) for $n\to\infty$ the uncertainty in the upper as well as the lower value becomes smaller and smaller and vanishes in the limit where the upper and lower values are then both equal to  $1$. Hence, the uncertainty as well as the risk disappears for $n\to\infty$ and the stochastic models converge to the constant $1$. In (b) and (c) this convergence is due to a concentration on the upper or the lower part which is converging to $1$ (and the corresponding probability interval goes to $\{1\}$).

Let us now explain how we avoid such convergences to the riskless asset by Assumption \ref{avoid_riskless_limit}.
 By Assumption \ref{ass34} we have $d^n<\min(D^n, u^n)\leq \max(D^n, u^n)<U^n$ with $d^n<1$ and $U^n>1$, for all $n\geq1$. Hence, by Assumption~\ref{par_limits} it always holds that $d\leq \min(D,u)\leq \max(D,u)\leq U$. Suppose now that Assumption~\ref{avoid_riskless_limit} (i) would not hold, then $d=U$. Then it has to hold that  $d=U=1$ and therefore $d=D=u=U=1$. Thus we are in the situation of case (a) above which we would like to exclude. Moreover, Assumption \ref{avoid_riskless_limit} (ii) guarantees that either the down parts do not converge to $1$ (then there is no need to restrict the probability interval further) or, if all down parts converge to $1$, then $\limsup\Pi^n>0$ allows to choose a subsequence where $\Pi^n\to\beta>0$ and therefore the up parts do not disappear in the limit. This avoids situation (b). For Assumption \ref{avoid_riskless_limit} (iii) the explanation is similar to (ii) only with reversed roles of up and down parts, avoiding situation (c).

\end{enumerate}
\end{remark}

Let us now define the  sets $\mathcal{P}^n$ which describe the possible choices of probabilities reflecting the uncertainty on each market $n$.

\begin{definition}\label{uncertainty_sets_n_1} Let
   \begin{equation}
\mathcal{R}^n =\bigl\{R\in\mathcal{M}_1:  R\circ (Y^n_{1})^{-1}\in \mathcal{L}^n_{1}\bigr\}.\label{supported_alt_0n_1}\end{equation}
\begin{equation}\label{uncer_set_n_1}
\mathcal{P}^n =\sigma\text{-conv}(\mathcal{R}^n)=\biggl\{P=\sum_{k=1}^\infty\alpha_kR_k :R_k\in\mathcal{R}^n, 0\leq \alpha_k\leq1, \sum_{k=1}^{\infty}\alpha_k=1\biggr\}\end{equation}
\end{definition}

\begin{remark}\label{importantremark}
We make the following observations.
\begin{enumerate}
\item
    The uncertainty set of \cite{blanch_carassus}
    (for $t=T=1$) is given as the set $\mathcal{P}^n$ of all convex combinations of measures in $\mathcal{R}^n$.
For our purposes we will allow  $\sigma$-convex-combinations as well. Hence, we define the uncertainty set ${\pp}^n$ as above. Note that in the notation of \cite{klein-koestenberger-robust-duality} the set  $\mathcal{P}^n$ is equal to the set ${\widetilde{\pp}}^n$.
\item  We get by \cite{blanch_carassus}
  that NA$(\rrr^n)$ and  NA$(\pp^n)$  hold for each $n\geq 1$. See also \cite{klein-koestenberger-robust-duality} (Lemma~55) for an easy proof for our case of the robust one-period model.
\item Observe that Assumption \ref{ass34} still allows for various combinations of no arbitrage and arbitrage for particular choices of measures $P\in\pp^n$, see 
  \cite{klein-koestenberger-robust-duality} (Lemma~56). This means that the strong NA condition of \cite{blanch_carassus} 
still does not hold (compare this to Lemma~4.4 therein).

\end{enumerate}
\end{remark}

 \begin{remark}\label{importantremark2}  The results in this section extensively use the theory developed in Section \ref{sec:RLFM}.
  This requires Assumption~\ref{closed}, i.e., the robust duality $l(\LL(\mathcal{P}^n))=E_n'$ for all $n\geq1$.
To ensure this property, \cite{klein-koestenberger-robust-duality} introduced the so-called \emph{Hahn-Extension} of the $\sigma$-algebra $\mathcal{F}$. 
In our setting, it is given by
\begin{equation*}
  \mathcal{H}_{\mathcal{F}}^{\widetilde{ \mathcal{R} }^{n}} = \sigma \biggl( \mathcal{F} \cup \biggl\{ \bigcup_{R\in \widetilde{\mathcal{R}}^{n}} E_{R}: E_{R} \subseteq S_{R}, E_{R}\in \mathcal{F} \biggr\} \biggr),
\end{equation*}
where $\widetilde{ \mathcal{R} }^{n} \subseteq \mathcal{R}^{n}$ is a special set of pairwise singular measures,\footnote{The expert reader may note that $\mathcal{R}^{n}$ already is a so-called supported alternative to $\pp^{n}$, in the sense of \cite{liebrich2022}. However, it was shown in \cite{klein-koestenberger-robust-duality} that $\mathcal{R}^{n}$ is not disjointly supported. The set $\widetilde{ \mathcal{R} }^{n}$ is a disjointly supported alternative to $\mathcal{R}^{n}$.} and $S_{R} = \{Y_{1} = d_{R}\} \cup \{Y_{1} = u_{R}\}$ is the \emph{support} of the measure $R$, see  \cite{klein-koestenberger-robust-duality} (Section 5.3) for details. 
It is important to note that $\widetilde{ \mathcal{R} }^{n}$ is in general not countable, and hence $\mathcal{H}_{\mathcal{F}}^{\widetilde{ \mathcal{R} }^{n}}$ can (and usually will) contain sets that are not in $\mathcal{F}$. 
The measures $\pp^{n}$ are readily extended to $\mathcal{H}_{\mathcal{F}}^{\widetilde{ \mathcal{R} }^{n}}$, since for each $P\in \pp^{n}$ there can be at most countably many $R_{k}'s$ in $\widetilde{ \mathcal{R} }^{n}$, such that $P(S_{R_{k}}) \neq 0$, and for every $A\in \mathcal{H}_{\mathcal{F}}^{\widetilde{ \mathcal{R} }^{n}}$ and $R\in \widetilde{ \mathcal{R} }^{n}$, we have that $A\cap S_{R} \in \mathcal{F}$ see \cite{klein-koestenberger-robust-duality} (Lemma~21). 
This is summarized in the following definition.\end{remark}

\begin{definition}\label{extended_model}
We define the $\sigma$-algebra $\widetilde{\mathcal{F}}^n  := \mathcal{H}_{\mathcal{F}}^{\widetilde{ \mathcal{R}}^n }$ to be the Hahn-extension of $\mathcal{F}$ for the localization $\widetilde{ \mathcal{R} }^n$ (as referred to in Remark~\ref{importantremark2}  above). 
We denote the extension of $\pp^{n}$ to $\widetilde{ \mathcal{F} }^{n}$ by $\widetilde{ \pp}^{n}$. 
\end{definition}

We may now apply \cite{klein-koestenberger-robust-duality} (Lemma~ and Corollary 46), which shows that, for all $n\geq 1$, the robust one-period binomial model based on 
$(\Omega^n, \widetilde{\mathcal{F}}^n,\widetilde{ \pp }^{n})$ is Hahn-localizable, and $\LL(\widetilde{ \pp }^{n})$ satisfies the robust $L^{1}$-$L^{\infty}$ duality by \cite{klein-koestenberger-robust-duality} (by Theorem~4), i.e., $(E_n)'=l(\mathbb{L}^\infty( \widetilde{ \pp}^n))$ is satisfied. Details can be found in \cite{klein-koestenberger-robust-duality}.  

Observe that $\mathrm{NAA1}(\pp^n)$ and $\mathrm{NAA2}(\pp^n)$  hold if and only if $\mathrm{NAA1}(\widetilde{ \pp }^{n})$ and $\mathrm{NAA2}(\widetilde{ \pp }^{n})$ hold respectively. 
        To see this, note that the events in Definition \ref{AA1} and Definition \ref{AA2} are already in $\mathcal{F}$, and hence the extended probability measures on $\widetilde{\mathcal{F}}^n$ agree with the original probability measures on $\mathcal{F}$ (see \cite{klein-koestenberger-robust-duality} for details).
Hence, we will slightly abuse notation, and write $\pp ^{n}$ for $\widetilde{ \pp }^{n}$.

  Our aim is now to characterize asymptotic arbitrage of the first and second  kind on the large financial market $\left((S^n_t)_{t=0,1}\right)_{n\geq1}$ in terms of conditions on the model parameters.

To this end we will formulate four theorems.  
The results are different for $U<+\infty$ and $U=+\infty$. In the second case the upper values in the sequence of robust binomial models
are allowed to become arbitrarily large, which, obviously, has an influence on the characterization of asymptotic arbitrage properties.

In the rest of this section we will always suppose that the set of conditions formulated in Assumption \ref{assumptions} holds
without further mentioning them.

\begin{theorem}\label{thm_U0_finite}
 Let $U<+\infty$. Then the following properties are equivalent:
 \begin{enumerate}
     \item NAA1$(\pp^n)$.
     \item $d<1<U$.
     \item For each sequence $(P^n)_{n\geq1}$ with $P^n\in \mathcal{P}^n$, for all $n$, there exists a sequence
 $(Q^n)_{n\geq1}$ with $Q^n\in \mathcal{Q}^n$, for all $n$, such that $(P^n)_{n\geq1}\triangleleft (Q^n)_{n\geq1}$.
 \end{enumerate}
\end{theorem}

\begin{theorem}\label{thm_U0_finite_AA2}
 Let $U<+\infty$.   Then the following properties are equivalent:
 \begin{enumerate}
     \item NAA2$(\pp^n)$.
      \item $D\leq 1\leq u$ and 
      \begin{enumerate}
          \item[(a)] if $d<1$ then $\liminf\pi^n>0$,
          \item[(b)] if $d=D=1$ then $\max(D^n-1,0)=o(1-d^n)$,
          \item[(c)] if $U>1$ then $\limsup\Pi^n<1$,
          \item[(d)] if $u=U=1$ then $\max(1-u^n, 0)=o(U^n-1)$.
      \end{enumerate}      
     \item For each sequence $(P^n)_{n\geq1}$ with $P^n\in \mathcal{P}^n$, for all $n$, we have that
 $(\mathcal{Q}^n)_{n\geq1}\triangleleft_w(P^n)$.
 \end{enumerate}
\end{theorem}

\begin{theorem}\label{thm_U0_infinite}
 Let $U=+\infty$. Then the following properties are equivalent:
 \begin{enumerate}
     \item NAA1$(\pp^n)$.
     \item $\limsup{\Pi^n}=0$, $d<1$ and $D<+\infty$.
     \item For each sequence $(P^n)_{n\geq1}$ with $P^n\in \mathcal{P}^n$, for all $n$, there exists a sequence
 $(Q^n)_{n\geq1}$ with $Q^n\in \mathcal{Q}^n$, for all $n$, such that $(P^n)_{n\geq1}\triangleleft (Q^n)_{n\geq1}$.
     \end{enumerate}  
\end{theorem}

\begin{theorem}\label{thm_U0_infinite_AA2}
 Let $U=+\infty$. Then the following properties are equivalent:
 \begin{enumerate}
     \item NAA2$(\pp^n)$.
     \item $\limsup{\Pi^n}<1$, $D\leq 1$ and, if $d=D=1$, then $\max(D^n-1,0)=o(1-d^n)$.
     \item For each sequence $(P^n)_{n\geq1}$ with $P^n\in \mathcal{P}^n$, for all $n$, we have that
 $(\mathcal{Q}^n)_{n\geq1}\triangleleft_w(P^n)$.
          \end{enumerate}  
\end{theorem}

\begin{remark}\begin{enumerate}\item Observe that the characterizations (i) $\Leftrightarrow$ (ii) in all four theorems do \emph{not} use the theory of \cite{klein-koestenberger-robust-duality}, hence the proofs work without the Hahn-extension $\widetilde{\mathcal{F}}^n$ given in Definition~\ref{extended_model}. 
Nevertheless, the assertions remain true in the extended setting (and are also  formulated  like this in the above theorems).
       
        Indeed, by assumption, $Y^n_1:(0,+\infty)\to(0,+\infty)$ is  $\mathcal{F}$ measurable. By enlarging $\mathcal{F}$ to $\widetilde{\mathcal{F}}^n$ we still have that $Y^n_1$ is $\widetilde{\mathcal{F}}^n$ measurable. Moreover, all sets $A$ involved in asymptotic arbitrage considerations are given by finite combinations of sets of the form  $A=\{Y^n_1=x\}$, $x\in(0,+\infty)$. Sets $A$ of this form are always contained in the original $\sigma$-algebra anyway, i.e., $A\in\mathcal{F}$. So there is no difference in these proofs.
        \item Observe that the implications (ii) $\Leftrightarrow$ (iii)  are not obvious in all four  theorems. However, Theorem~\ref{lfmftap} and Theorem~\ref{lfmftap2} can be applied here to show that (i) $\Leftrightarrow$ (iii)  as $l(\mathbb{L}^{\infty}(\pp^n))=E_n'$ for the proper choice of the $\sigma$-algebra $\widetilde{\mathcal{F}}^n  = \mathcal{H}_{\mathcal{F}}^{\widetilde{ \mathcal{R}}^n }$ and the extended set $\pp^n$ as given in Definition~\ref{extended_model}. Together with the direct proofs of  (i) $\Leftrightarrow$ (ii) the equivalences of all three conditions follow.
        \end{enumerate}
    
\end{remark}
We postpone the proofs of the above theorems to Subsection~\ref{proofs_AA} and first study the consequences of the theorems on possible limit models in Subsection~\ref{limit model}. In the remainder of this section we will repeatedly  use 
the following notation.

\begin{notation}\label{nota_d_u}
Let $R\in\mathcal{R}^n$. The law of $Y^n_1$ under $R$ is a binomial distribution. We will  use the notation $u_R$, $d_R$, $\pi_R$  for the parameters corresponding to $R$, where $u_R$, $d_R$ denote the up and down value corresponding to $R$ and $\pi_R=R(Y^n_1=u_R)$, where  $(u_R,d_R,\pi_R)\in I^n=[u^n,U^n]\times[d^n,D^n]\times [{\pi}^n,{\Pi}^n]$.
\end{notation}

 \subsection{NA of robust limit models and its relation to NAA1 and NAA2}\label{limit model}

Let us first clarify what we mean by a robust limit model. We already assumed in Assumption~\ref{par_limits} that the parameters $d^n$, $D^n$, $u^n$, $U^n$ converge to $d\in[0,1]$, $D\in[0,+\infty]$, $u\in[0,+\infty]$, $U\in[1,+\infty]$, respectively. Observe that we did not assume convergence of $\pi^n, \Pi^n$, which we will now do to properly define a robust limit model.

\begin{definition}[Robust limit model]\label{def_limmitmod}

Assume that there exists $\pi,\Pi\in[0,1]$ with $\lim_{n \to\infty}\pi^n=\pi$ and
$\lim_{n \to\infty}\Pi^n=\Pi$. Then the robust limit binomial model is analogously defined as the robust binomial model for each $n$ but with the additional degenerate possible cases $U=+\infty$, $d=0$, $\pi=0$ or $=1$ and $\Pi=0$ or $=1$. In particular, this means, that in the limit the bijective random variable $Y_1$ satisfies $Y_1: [0,+\infty]\to[0,+\infty]$ with $\Omega=[0,+\infty]$ and $\mathcal{F}=\mathcal{B}[0,+\infty]$. The (discounted) price process of the risky asset in the limit model is given by 
 $S_0=1$ and  $$S_{1}=S_0\cdot Y_{1}.$$
The sets $\mathcal{R}$ and $\mathcal{P}$ for the limit model are defined analogously as in Definition~\ref{uncertainty_sets_n_1} with $\mathcal{M}_1$ now defined with probability measures on $[0,+\infty]$. \end{definition} 

Analogously to Notation~\ref{nota_d_u}, in the rest of this subsection, we will use the notation $(u_R,d_R,\pi_R)\in I$, where $I=[u,U]\times[d,D]\times[\pi,\Pi]$, for the parameters corresponding to  a measure $R\in\rrr$ .

The following lemma characterizes the robust no arbitrage property in the limit  in terms of the parameters of the limit model (including degenerate parameter values).
\begin{lemma}\label{par_NA_limit}
NA$(\mathcal{P})$ holds for the limit model if and only if the three following conditions are satisfied.
\begin{enumerate}
    \item $0\leq d<1<U\leq +\infty$,
    \item  $[\pi,\Pi]=\{0\}$ $\Rightarrow$ $D>1$,
    \item  $[\pi,\Pi]=\{1\}$ $\Rightarrow$ $u<1$.\end{enumerate}

  \end{lemma}
Lemma~\ref{par_NA_limit} and Theorem~\ref{thm_U0_finite} immediately have the following consequence which implies that  in the non-degenerate limit case of $U<+\infty$ and $[\pi,\Pi]\ne\{0\}$ and $\ne \{1\}$ we can characterize no arbitrage in the robust limit model exactly by NAA1 of the sequence of robust binomial models.
\begin{theorem}[non-degenerate limit model]\label{nondegen_limit}
Let $U<+\infty$ and $[\pi,\Pi]\neq \{0\}$ and $\neq\{1\}$. Then the conditions NA$(\mathcal{P})$ and NAA1$(\mathcal{P}^n)$ are equivalent.   
\end{theorem}

The easy proofs of  Lemma~\ref{par_NA_limit} and Theorem~\ref{nondegen_limit} are postponed to the end of this subsection. 

By choosing appropriate convergent sequences of parameters we will now give four examples that show that both conditions  NAA1$(\mathcal{P}^n)$, NAA2$(\mathcal{P}^n)$ {do, in general, not imply} NA$(\mathcal{P})$ in the limit model and, on the other hand, NA$(\mathcal{P})$ in the limit model {does, in general, not imply } NAA1$(\mathcal{P}^n)$ or NAA2$(\mathcal{P}^n)$. It is well known that NAA1$(\mathcal{P}^n)$ and NAA2$(\mathcal{P}^n)$ in general do not imply each other, hence we see that, in general, all three conditions are different.

\begin{example}[NAA1$(\mathcal{P}^n)$ holds, but there is an arbitrage in the limit model]
Choose convergent sequences $d^n, D^n, u^n, U^n, \pi^n, \Pi^n$ (which satisfy all conditions of all models $n$) such that the limits satisfy $\pi=\Pi=0$, $d<1=D<U\leq +\infty$. For $U<+\infty$ Theorem~\ref{thm_U0_finite} implies that NAA1$(\mathcal{P}^n)$ holds (as $d<1<U$). For $U=+\infty$ Theorem~\ref{thm_U0_infinite} implies that NAA1$(\mathcal{P}^n)$ holds (as $\lim\Pi^n=0$, $d<1$, $D<+\infty$).
As $1=D$ Lemma~\ref{par_NA_limit} implies that there has to exist arbitrage in the limit model. Indeed, it is also  easy to find directly, i.e., $X_1=-(S_1-S_0)=1-Y_1$. As $D=1$, $X_1\geq0$, $\mathcal{P}$-q.s. and $R(X_1>0)=R(Y_1=d)=1$ for $R$ with the parameters $d_R=d$, $\pi_R=0$ and $u_R\in[u,U]$ arbitrary. 

\end{example}

\begin{example}[NA$(\mathcal{P})$ in the limit model holds but there is an AA1$(\mathcal{P}^n)$]
Choose convergent sequences $d^n, D^n, u^n, U^n, \pi^n, \Pi^n$ (which satisfy all conditions of all models $n$) such that the limits satisfy $0<\Pi<1$, $U=+\infty$
and $d\leq D<1$. In particular we have that $d<1<U$ and $[\pi,\Pi]\ne \{0\}, \{1\}$. Hence, by Lemma~\ref{par_NA_limit}, NA$(\mathcal{P})$ holds.
On the other hand Theorem~\ref{thm_U0_infinite} implies that there exists an 
AA1$(\mathcal{P}^n)$ as $U=+\infty$ and $\lim \Pi^n=\Pi>0$.
\end{example}

\begin{example}[NAA2$(\mathcal{P}^n)$ holds, but there is an arbitrage in the limit model]
We choose appropriate convergent sequences $d^n, D^n, u^n, U^n, \pi^n, \Pi^n$ (which satisfy all conditions of all models $n$). Take, e.g.,  $d=D=\frac13$, $u=U=1$, $\pi=\Pi=\frac12$. Then the limit model ist just a standard Binomial model with up value $1$ and down value $\frac13$. It is well known that this model has arbitrage.  Choose $u^n=1-\frac1{n^2}$ and $U^n=1+\frac1{n}$ which is consistent with  the limits $u=U=1$. Check now the equivalent condition of Theorem~\ref{thm_U0_finite_AA2} for NAA2$(\mathcal{P}^n)$. It holds that $d=D=\frac13<1\leq u=1$. As $\pi=\frac12$, we have that (iia) holds. As $u=U=1$ we have to check (iid): clearly $1-u^n=\frac1{n^2}$ is $o(U^n-1)$, as $U^n-1=\frac1n$. Hence NAA2$(\mathcal{P}^n)$ holds.
\end{example}   

\begin{example}[NA$(\mathcal{P})$ in the limit model holds but there is an AA2$(\mathcal{P}^n)$]
We choose appropriate converging sequences $d^n, D^n, u^n, U^n, \pi^n, \Pi^n$ (which satisfy all conditions of all models $n$).
Choose, e.g., $d=\frac12$, $D=\frac32$, $u\geq1$, $U=2$, $\pi=\Pi=\frac12$. Then the limit model satisfies NA$(\mathcal{P})$. Indeed,
$d<1<U$ and $[\pi,\Pi]=\{\frac12\}\ne\{0\},\{1\}$ and so NA$(\mathcal{P})$ follows by Lemma~\ref{par_NA_limit}.
But as $D=\frac32>1$ we see by Theorem~\ref{thm_U0_finite_AA2} that there exists an AA2$(\mathcal{P}^n)$.

\end{example}

\begin{proof}[Proof of Lemma~\ref{par_NA_limit}]

For $R \in \mathcal{R}$ define $\mu^R=R \circ Y_1^{-1}$. 
Then, by Proposition 5.7 in 
\cite{blanch_carassus}
the following assertions are equivalent: 
\begin{enumerate}
\item[a)]
NA$(\mathcal{P})$ holds,
\item[b)] $1 \in \text{ri conv} (\cup_{R} \operatorname{supp} \mu^R)$, where $\text{ri conv}$ denotes the relative interior of the convex hull of a set. 
\end{enumerate}
The assertion of the lemma thus follows if we prove the equivalence of (i), (ii), (iii) with b). 
If $\Pi \neq 0$ and $\pi \neq 1$, then $\text{conv}(\cup_{R} \operatorname{supp} \mu^R) =[d, U]$. Therefore, 
$1 \in \text{ri conv} (\cup_{R} \operatorname{supp} \mu^R)$ is equivalent to $d < 1 < U$. 

If $\Pi=0$ (and thus $\pi=0$)
then $\text{conv}(\cup_{R} \operatorname{supp} \mu^R) =[d, D]$. Therefore, 
$1 \in \text{ri conv} (\cup_{R} \operatorname{supp} \mu^R)$ is equivalent to $d < 1 < D$. Assumption~\ref{ass34} this implies that $U>1$, too.

If $\pi=1$ (and thus $\Pi=1$)
then $\text{conv}(\cup_{R} \operatorname{supp} \mu^R) =[u, U]$. Therefore, 
$1 \in \text{ri conv} (\cup_{R} \operatorname{supp} \mu^R)$ is equivalent to $u < 1 < U$. Assumption~\ref{ass34}  implies that $d < 1$, too. 
\end{proof}

\begin{proof}[Proof of Theorem~\ref{nondegen_limit}]
 Assume NA$(\mathcal{P})$ holds. As, by assumption $[\pi,\Pi]\neq \{0\}$ and $\neq\{1\}$, we get by Lemma~\ref{par_NA_limit} that $d<1<U$. As, by assumption $U<+\infty$, Theorem~\ref{thm_U0_finite} implies NAA1$(\mathcal{P}^n)$. 

 Assume now that NAA1$(\mathcal{P}^n)$ holds. As $U<+\infty$ Theorem~\ref{thm_U0_finite} implies $d<1<U$. By assumption $[\pi,\Pi]\neq \{0\}$ and $\neq\{1\}$ hence Lemma~\ref{par_NA_limit} implies NA$(\mathcal{P})$.
\end{proof}

\subsection{The proofs of Theorems~\ref{thm_U0_finite}, \ref{thm_U0_finite_AA2}, \ref{thm_U0_infinite}, \ref{thm_U0_infinite_AA2}}\label{proofs_AA}

The proofs consist of various lemmas which follow here.

\begin{lemma}\label{AA1_all_sets} 
NAA1$(\pp^n)$ $\Leftrightarrow$ NAA1$(\rrr^n)$ and NAA2$(\pp^n)$ $\Leftrightarrow$NAA2$(\rrr^n)$.
\end{lemma}

\begin{proof}
    As $\rrr^n\subset \pp^n$ with the same polar sets it is clear that NAA1$(\pp^n)$ $\Rightarrow$ NAA1$(\rrr^n)$ and NAA2$(\pp^n)$ $\Rightarrow$ NAA2$(\rrr^n)$.

We will now show NAA1$(\rrr^n)$ $\Rightarrow$ NAA1$(\pp^n)$. Indeed, assume that NAA1$(\rrr^n)$ holds but there would exist an AA1$(\pp^n)$. That is, there exists a subsequence of markets, still denoted by $n$, such that there exists $H^n\in\mathbb{R}$ such that for $X^n_1=H^n(Y^n_1-1)$ we have that (i) $X^n_1\geq -c_n$, $\pp^n$-q.s., with $\lim_{n\to\infty}c_n=0$, and (ii) there exists $\alpha>0$ and $P^n\in\pp^n\setminus \rrr^n$ with $P^n(X^n_1\geq\alpha)\geq\alpha$, for all $n$.
It is clear that (i) implies that $X^n_1\geq -c_n$, $\rrr^n$-q.s., for all $n$. By the definition of $\pp^n$ we have that $P^n=\sum_{k=1}^{\infty}\lambda^n_kR^n_k$ with $R^n_k\in\rrr^n$, for all $k\geq1$, and $0\leq \lambda_k\leq1$, $\sum_{k=1}^{\infty}\lambda_k=1$. As $P^n(X^n_1\geq\alpha)\geq\alpha$, for all $n$, we have that, for each $n$, there exists $k_n\geq 1$, such that $\lambda^n_{k_n}>0$ and $R^n_{k_n}(X^n_1\geq\alpha)\geq\alpha$. (If not then (ii) is not satisfied.) For the sequence $R^n:=R^n_{k_n}$ this gives an AA1$(\rrr^n)$, a contradiction.

The proof of NAA2$(\rrr^n)$ $\Rightarrow$ NAA2$(\pp^n)$ follows analogously.
\end{proof}

\begin{lemma}\label{NAA1_1>d_0}
Assume that NAA1$(\pp^n)$ holds. Then $d<1$. 
\end{lemma}

\begin{proof}
By Assumption \ref{ass34}, 
we have that $d\leq D$. If $D<1$, it also follows that $d\leq D<1$. Assume now that $D\geq1$ and $d=1$. 
Define, for each $n$,  $\ep_n=1-d^n$. Note that $\ep_n>0$ since we assumed $d^n<1$. 
As $d=1$, we have that $\ep_n\to0$. 
Let $X^n_1=Y^n_1-1$. 
Observe that \ref{ass34}
implies  that $u_R\geq u^n> d^n$. By definition $d_R\geq d^n$. 
Therefore, for all $n$ and $R\in \rrr^n$ we have $R(X^n_1\geq -\ep_n)=1$, and hence, $X^n_1\geq -\ep_n$ $\pp^n$-q.s.
Define now the measure $P^n\in\rrr^n$ as follows: $P^n\circ(Y^n_1)^{-1}=\Pi^n\delta_{U^n}+(1-\Pi^n)\delta_{D^n}$. 
Now, we distinguish the following two cases: (1) $D>1$, and (2) $D=1$. In case  (1) observe that, by \ref{ass34},  $U^n>D^n$.  As $D>1$, by Assumption~\ref{par_limits} there exists $\alpha>0$ and $n_0\in\mathbb{N}$ such that, for all $n\geq n_0$, $U^n> D^n\geq 1+\alpha$.  It holds that $P^n(X^n_1\geq\alpha)=1$ for all $n$ as $U^n-1\geq\alpha$ and $D^n-1\geq\alpha$.
This gives an AA1 (even a strong asymptotic arbitrage), a contradiction. 
In case (2) we have that $d=D=1$ and $D^n\to1$, for $n\to\infty$. Assumption~\ref{avoid_riskless_limit} (ii) implies that $\limsup\Pi^n>0$ 
Choose a subsequence of markets (still denoted by $n$) with $\Pi^n\to\beta$, where $0<\beta\leq 1$. 
Further, note that Assumption~\ref{avoid_riskless_limit} (i) implies that $U>1=d$.
Therefore, by Assumption~\ref{par_limits}, there exists $\alpha>0$ and $n_0\in\mathbb{N}$ such that, for all $n\geq n_0$, $U^n\geq 1+\alpha$. It holds that $P^n(X^n_1\geq\alpha)=P^n(Y^n_1=U^n)=\Pi^n\to\beta>0$, for $n\to\infty$.  This gives an AA1, a contradiction. Therefore, in both cases $d<1$ has to hold.
\end{proof}

\begin{lemma}\label{NAA1_1<U_0}
Assume that NAA1$(\pp^n)$ holds. Then $U>1$. 
\end{lemma}

\begin{proof} 
  By Assumption \ref{ass34}, 
  for all $n$, we have that $u^n< U^n$. Therefore, if $u>1$, it holds that $U\ge u>1$. Assume now that $u\leq1$ and $U=1$. 
 Define, for each $n$,  $\ep_n=U^n-1$. Note that $\ep_n>0$ since we assumed $U^n>1$.
 As $U=1$, we have that $\ep_n\to0$, for $n\to\infty$. 
Let $X^n_1=-(Y^n_1-1)$. 
Observe that Assumption \ref{ass34}
implies that, for each $R\in\rrr^n$ and $n\geq 1$, we have that $d_R\leq D^n< U^n=1+\ep_n$. By definition $u_R\leq U^n$. Therefore, for all $n$, $R(X^n_1\geq -\ep_n)=1$, for all $R\in\rrr^n$, hence $\pp^n$-q.s.  Now, choose the measure $P^n\in\rrr^n$ such that $P^n\circ(Y^n_1)^{-1}=\pi^n\delta_{u^n}+(1-\pi^n)\delta_{d^n}$.
We distinguish the following two cases: (1) $u<1$, (2) $u=1$. In case  (1) observe that, by \ref{ass34},
$d^n<u^n$.  As $u<1$, by Assumption~\ref{par_limits} there exists $\alpha>0$ and $n_0\in\mathbb{N}$ such that, for all $n\geq n_0$, $d^n\leq u^n\leq 1-\alpha$.  It holds that $P^n(X^n_1\geq\alpha)=1$ for all $n$ as $1-u^n\geq\alpha$ and $1-d^n\geq\alpha$.
This gives an AA1 (even a strong asymptotic arbitrage), a contradiction. 
In case (2) we have that $u=U=1$ and $u^n\to1$, for $n\to\infty$. Assumption~\ref{avoid_riskless_limit} (iii) 
implies that $\liminf\pi^n<1$. Choose a subsequence of markets (still denoted by $n$) with $\pi^n\to\beta$, where $0\leq\beta<1$. Further, note that Assumption~\ref{avoid_riskless_limit} (i) implies that $U=1>d$ 
Therefore, by Assumption~\ref{par_limits}, there exists $\alpha>0$ and $n_0\in\mathbb{N}$ such that, for all $n\geq n_0$, $d^n\leq 1-\alpha$.   It holds that $P^n(X^n_1\geq\alpha)=P(Y^n_1=d^n)=1-\pi^n\to1-\beta>0$, for $n\to\infty$.  This gives an AA1, a contradiction. Therefore in both cases $U>1$ has to hold.
\end{proof}

\begin{lemma}\label{NAA1_reverse_finiteU0}
 Let $d<1<U<+\infty$. Then NAA1$(\rrr^n)$ holds.
 \end{lemma}

 \begin{proof}[Proof of Lemma~\ref{NAA1_reverse_finiteU0}]
     Suppose, by contradiction, that there exists an AA1$(\rrr^n)$,  that is, there exists a subsequence of markets, still denoted by $n$, such that there exists $H^n\in\mathbb{R}$ such that for $X^n_1=H^n(Y^n_1-1)$ we have that (i) $X^n_1\geq -c_n$, $\rrr^n$-q.s., with $\lim_{n\to\infty}c_n=0$, and (ii) there exists $\alpha>0$ and $P^n\in\rrr^n$ with $P^n(X^n_1\geq\alpha)\geq\alpha$, for all $n$. As $d<1<U<+\infty$, there exists $n_0$ and $0<\beta<1$, $0<\gamma<1$, such that, for all $n\geq n_0$, we have that $0<d^n<1-\beta$ and $1+\gamma<U^n<2U<+\infty$. As $c_n\to0$, for each $\ep>0$ there exists $n(\ep)\geq n_0$ such that for all $n\geq n(\ep)$, $c_n\max(\frac{2U-1}{\beta},\frac{1}{\gamma})<\ep$. W.l.o.g. we can assume that $H^n\ne 0$, for all $n$, as $H^n=0$ does not contribute to the AA1. Define $N_1=\{n\geq 1: H^n>0\}$ and $N_2=\{n\geq 1: H^n<0\}$. Both sets $N_1$, $N_2$, can be unbounded, at least one of them has to be unbounded. If one of the two sets would not be unbounded then it does not contribute to the AA1 substantially as only $n\to\infty$ plays an important role. Therefore, assume that $N_1$, $N_2$ both are unbounded. Let $n\in N_1$ with $n\geq n_0$. By (i), for each $R\in\rrr^n$, it has to hold that $R(X^n_1\geq-c_n)=1$. As $H^n>0$, the smallest possible value of $X^n_1$ appears for $R^n\in\rrr^n$ with $d_{R^n}=d^n$. Then Assumption~\ref{ass_blan_car_1step} implies that $R^n(X^n_1=H^n(d^n-1))=1-\pi_{R^n}\geq 1-\pi^n>0$. As $0<d^n<1-\beta$ it has to hold that $H^n(d^n-1)=-H^n(1-d^n)\geq-c_n$ and thus
\begin{equation}\label{Hn+1}
 0<H^n\leq\frac{c_n}{1-d^n}< \frac{c_n}{\beta}. 
\end{equation}
Further, observe that, for all $R\in \rrr^n$, we have that $R(X^n_1\leq H^n(U^n-1))=1$ and therefore, as $1<U^n<2U$ and by (\ref{Hn+1}),
\begin{equation}\label{Hn+2}
 X^n_1\leq H^n(U^n-1)< c_n\frac{2U-1}{\beta}, \quad\text{$\rrr^n$-q.s.} 
\end{equation}
Let $\ep<\alpha$. Then, for $n\geq n(\ep)$, $c_n\frac{2U-1}{\beta}<\ep$. Hence, for  $n\in N_1$, $n\geq n(\ep)$, by (\ref{Hn+2}), $\{X^n_1\geq \alpha\}$ is an $\rrr^n$-polar set and there cannot exist $P^n\in\rrr^n$ as in (ii) for $n\in N_1$.

Let now $n\in N_2$ with $n\geq n_0$. By (i), for each $R^n\in\rrr^n$, it has to hold that $R^n(X^n_1\geq-c_n)=1$. As $H^n<0$, the smallest possible value of $X^n_1$ appears for $R^n\in\rrr^n$ with $u_{R^n}=U^n$. Then, by Assumption~\ref{ass_blan_car_1step}, $R^n(X^n_1=-|H^n|(U^n-1))=\pi_{R^n}\geq\pi^n>0$. Hence it has to hold that $-|H^n|(U^n-1)\geq-c_n$ and thus, as $1+\gamma<U^n$, 
\begin{equation}\label{Hn-1}
 0<|H^n|\leq\frac{c_n}{U^n-1}< \frac{c_n}{\gamma}. 
\end{equation}
Further, observe that, $X^n_1=-|H^n|(Y^n-1)$ is $<0$ if $Y^n-1>0$ and, if $Y^n-1\leq0$, then $X^n_1=|H^n|(1-Y^n_1)\geq 0$ with maximal value $|H^n|(1-d^n)$. Hence, for all $R\in \rrr^n$, we have that $R(X^n_1\leq |H^n|(1-d^n))=1$ and therefore, as $0<d^n<1-\beta$ and by (\ref{Hn-1}),
\begin{equation}\label{Hn-2}
 X^n_1\leq |H^n|(1-d^n)< c_n/\gamma, \quad\text{$\rrr^n$-q.s.} 
\end{equation}

Let  $\ep<\alpha$ Then, for  $n\in N_2$, $n\geq n(\ep)$, $\frac{c_n}{\gamma}<\ep$ and by (\ref{Hn-2}), $\{X^n_1\geq \alpha\}$ is an $\rrr^n$-polar set and there cannot exist $P^n\in\rrr^n$ as in (ii) for $n\in N_2$. Therefore there cannot exist an AA1($\rrr^n$).
 \end{proof}

\begin{proof}[Proof of Theorem~\ref{thm_U0_finite}]
(i) $\Rightarrow$ (ii) follows by Lemma~\ref{NAA1_1>d_0} and Lemma~\ref{NAA1_1<U_0}.  (ii) $\Rightarrow$ (i) follows by Lemma~\ref{NAA1_reverse_finiteU0} and Lemma~\ref{AA1_all_sets}. (i) $\Leftrightarrow$ (iii) then follows by Theorem~\ref{lfmftap}. This shows that (i), (ii), (iii) are equivalent.
\end{proof}

\begin{lemma}\label{NAA2_D_0<=1}
Assume that NAA2$(\pp^n)$ holds. Then $D\leq1$.
\end{lemma}

\begin{proof}
     Suppose we would have $1<D\leq +\infty$. Let $X^n_1=Y^n_1-1$. Then, for each $R\in\rrr^n$ we have that $Y^n_1$ can only take the values $u_R>0$ and $d_R>0$, with $u_R\in[u^n,U^n]$ and $d_R\in[d^n,D^n]$. As both values are strictly positive since we assumed $0<d^n<u^n$, see Assumption \ref{ass34}, we have that $R(X^n_1>-1)=1$, for all $R\in\rrr^n$, hence $X^n_1\geq-1$, $\pp^n$-q.s. Let $P^n\circ(Y^n_1)^{-1}=\Pi^n\delta_{U^n}+(1-\Pi^n      )\delta_{D^n}$. As $D^n\to D>1$ there exists $\alpha>0$ and $n_0\in\mathbb{N}$ such that, for all $n\geq n_0$, $D^n\geq 1+\alpha$. Since $U^n> D^n$, by Assumption \ref{ass34}, 
$$P^n(X^n_1\geq\alpha)= 1.$$ This gives an AA2.
\end{proof}

\begin{lemma}\label{NAA1+U_0<inf_u_0>=1}
Assume that $U<+\infty$ and NAA2$(\pp^n)$ holds. Then $u\geq1$.
\end{lemma}

\begin{proof} We proceed similarly as in the proof of Lemma~\ref{NAA2_D_0<=1}. 
     Suppose we would have $1>u\geq 0$. Let $X^n_1=-\frac1{U^n}(Y^n_1-1)$. Then,
  as $d_R\leq D^n<U^n$, by Assumption~\ref{avoid_riskless_limit} (i), for all $n$, and $u_R\leq U^n$ we have that $\frac{-a+1}{U^n}\geq-1$, for $a=d_R, u_R$. Hence  
     $R(X^n_1\geq-1)=1$, for all $R\in\rrr^n$, which implies $X^n_1\geq-1$, $\pp^n$-q.s. Let $P^n\circ(Y^n_1)^{-1}=\Pi^n\delta_{u^n}+(1-\Pi^n)\delta_{d^n}$. 
   As $d^n<u^n$, $u<1$, $U<+\infty$, and $U^n\to U$
     there exists $0<\alpha<1$ and $n_0\in\mathbb{N}$ such that, for all $n\geq n_0$, $d^n<u^n\leq 1-\alpha$ and $U^n\leq 2U<+\infty$. Define $\widetilde{\alpha}=\frac{\alpha}{2U}>0$. Therefore, for all $n\geq n_0$,
$$P^n(X^n_1\geq\widetilde{\alpha})= 1.$$ This gives an AA2.
\end{proof}

\begin{lemma}\label{AA2liminf_prob1}
Assume that $U<+\infty$, $d<1$ and NAA2$(\pp^n)$ holds. Then it follows that $\liminf\pi^n>0$.
\end{lemma}

\begin{proof}
    Assume that on the contrary, $\liminf \pi^n=0$. Choose a subsequence (again denoted by $n$) such that $\lim_{n\to\infty}\pi^n=0$. 
By assumption, $d<1$ and $U<+\infty$. Hence there exists $\alpha\in(0,1)$  and  $n_0\in\mathbb{N}$  such that $d^n\leq 1-\alpha$ and  $U^n\leq2U<\infty$ for all $n\geq n_0$. Define $X^n_1$ as in the proof of Lemma~\ref{NAA1+U_0<inf_u_0>=1}, that is, $X^n_1=-\frac{1}{U^n}\left(Y^n_1-1\right)$. This will give an AA2. Indeed, as in the proof of Lemma~\ref{NAA1+U_0<inf_u_0>=1} we have that
$X^n_1\geq -1$, ${\pp}^n$-q.s. Define $P^n$ as follows: $P^n\circ {Y^n_1}^{-1}=\pi^n\delta_{U^n}+(1-\pi^n)\delta_{d^n}$ and observe that $\frac1{U^n}(-d^n+1)\geq \frac{\alpha}{2U}=:\widetilde{\alpha}>0$, for all $n\geq n_0$. Therefore 
$$P^n(X^n_1\geq\widetilde{\alpha})=P(Y^n=d^n)=1-\pi^n\to 1,$$
an AA2, a contradiction. 
\end{proof}

\begin{lemma}\label{AA2liminf_prob4}
Assume that  $d=D=1$ and NAA2$(\pp^n)$ holds. Then  it follows $\max(D^n-1,0)=o(1-d^n)$.
\end{lemma}

\begin{proof} Assume that $\max(D^n-1,0)=o(1-d^n)$ does not hold. We can choose a subsequence (still denoted by $n$) such that $D^n-1>0$, for all $n$,  and $\lim_{n\to\infty}\frac{D^n-1}{1-d^n}=2\alpha$, for some $\alpha>0$.
As $D=d=1$, Assumption~\ref{avoid_riskless_limit} (i) implies $U>1$. Then there exists $n_0$ and $0<\beta<1$  such that, for all $n\geq n_0$, $U^n\geq 1+\beta$ and $\frac{D^n-1}{1-d^n}\geq\alpha$. Define $X^n_1=\frac{1}{1-d^n}(Y^n_1-1)$.  Then, for each $R\in\rrr^n$,
$X^n_1\geq \frac{1}{1-d^n}(d^n-1)=-1$, hence $X^n_1\geq-1$,  $\pp^n$-q.s. Let $P^n\circ(Y^n_1)^{-1}=\Pi^n\delta_{U^n}+(1-\Pi^n)\delta_{D^n}$.  As,  $U^n>D^n$ by Assumption \ref{ass34}, and as $\frac{D^n-1}{1-d^n}\geq\alpha$, for all $n\geq n_0$, we have that 
$$P^n(X^n_1\geq\alpha)=1,$$
for all $n\geq n_0$, which gives an AA2.
\end{proof}

\begin{lemma}\label{AA2liminf_prob3}
Assume that  $U>1$ and NAA2$(\pp^n)$ holds. Then  $\limsup\Pi^n<1$.
\end{lemma}

\begin{proof}
Assume that  $\limsup \Pi^n=1$. Choose a subsequence (again denoted by $n$) such that $\lim_{n\to\infty}\Pi^n=1$. By assumption $U>1$. Therefore there exists $\alpha\in(0,1)$  and  $n_0\in\mathbb{N}$  such that $U^n\geq 1+\alpha$  for all $n\geq n_0$. Define  $X^n_1=Y^n_1-1$ (as in the proof of Lemma~\ref{NAA2_D_0<=1}). As there we see that $X^n_1>-1$, $\pp^n$-q.s.
Let $P^n\circ {Y^n_1}^{-1}=\Pi^n\delta_{U^n}+(1-\Pi^n)\delta_{d^n}$. Then, for $n\geq n_0$ we have that 
$$P^n(X^n_1\geq\alpha)=P^n(Y^n_1=U^n)=\Pi^n\to 1,$$
for $n\to\infty$. This gives an AA2, a contradiction.
\end{proof}

\begin{lemma}\label{AA2liminf_prob2}
Assume that  $u=U=1$ and NAA2$(\pp^n)$ holds. Then  it follows $\max(1-u^n, 0)=o(U^n-1)$.
\end{lemma}

\begin{proof}
Assume that $\max(1-u^n, 0)=o(U^n-1)$ does not hold. We can choose a subsequence (still denoted by $n$) such that $1-u^n>0$, for all $n$,  and $\lim_{n\to\infty}\frac{1-u^n}{U^n-1}=2\alpha$, for some $\alpha>0$.
As $U=u=1$, Assumption~\ref{avoid_riskless_limit} (i) implies $d<1$. Then there exists $n_0$ and $0<\beta<1$  such that, for all $n\geq n_0$, $d^n\leq 1-\beta$ and $\frac{1-u^n}{U^n-1}\geq\alpha$, for all $n\geq n_0$.  Define $X^n_1=-\frac{1}{U^n-1}(Y^n_1-1)=\frac{1}{U^n-1}(1-Y^n_1)$.  Then, for each $R\in\rrr^n$,
$X^n_1\geq \frac{1}{U^n-1}(1-U^n)=-1$, hence $X^n_1\geq-1$,  $\pp^n$-q.s. Let $P^n\circ(Y^n_1)^{-1}=\Pi^n\delta_{u^n}+(1-\Pi^n)\delta_{d^n}$.  As, by \ref{ass34}, $u^n>d^n$, and as $\frac{1-d^n}{U^n-1}>\frac{1-u^n}{U^n-1}\geq\alpha$, for all $n\geq n_0$, we have that 
$$P^n(X^n_1\geq\alpha)=1,$$
for all $n\geq n_0$, which gives an AA2.
\end{proof}

\begin{lemma}\label{NAA2_reverse_finiteU0}
 Let $U<+\infty$ and assume that (ii) of Theorem~\ref{thm_U0_finite_AA2} holds. Then NAA2$(\rrr^n)$ holds.
 \end{lemma}

 \begin{proof}[Proof of Lemma~\ref{NAA2_reverse_finiteU0}]
Suppose, to the contrary, that there would exist an AA2$(\rrr^n)$,  that is, there exists a subsequence of markets, still denoted by $n$, such that there exists $H^n\in\mathbb{R}$ such that for $X^n_1=H^n(Y^n_1-1)$ we have that (i) $X^n_1\geq -1$, $\rrr^n$-q.s., and (ii) there exists $\alpha>0$ and $P^n\in\rrr^n$ with $P^n(X^n_1\geq\alpha)\geq1-\ep_n$,  where $\lim_{n\to\infty}\ep_n=0$.      
Define, similarily as in the proof of Lemma~\ref{NAA1_reverse_finiteU0},  $N_1=\{n\geq 1: H^n>0\}$ and $N_2=\{n\geq 1: H^n<0\}$ and, as there, assume  that $N_1$, $N_2$ both are unbounded. 

Let $n\in N_1$. 
As $H^n>0$, the smallest possible value of $X^n_1$ appears for $R^n\in\rrr^n$ with $d_{R^n}=d^n$. Then $R^n(X^n_1=H^n(d^n-1))=1-\pi_{R^n}\geq 1-\pi^n>0$ by Assumption~\ref{ass_blan_car_1step}. Hence, it has to hold that $H^n(d^n-1)=-H^n(1-d^n)\geq-1$ and thus
\begin{equation}\label{Hn+1A2}
 0<H^n\leq\frac{1}{1-d^n}. 
\end{equation}
We will distinguish the following cases: (1) $d<D\leq1$ and (2) $d=D=1$.

Case (1): As $d<1$ there exists $0<\beta<1$ and $n_0$ such that, for all $n\geq n_0$, $d^n\leq 1-\beta$. As $D\leq 1$, for each $\ep>0$ there exists $n(\ep)\geq n_0$, such that, for all $n\geq n(\ep)$,  we have that $D^n\leq 1+\ep$. Observe that, for each $R^n\in\rrr^n$, with $d_{R^n}\in[d^n,D^n]$, for all $n\geq n(\ep)$, we have that 
$$\{Y^n_1=d_{R^n}\}\subseteq \{X^n_1\leq \frac{\ep}{\beta}\},$$
as $H^n(d_{R^n}-1)\leq H^n(D^n-1)\leq \frac{\ep}{\beta}$, by (\ref{Hn+1A2}).  Hence, (ii) cannot happen on the down parts, as, for $\ep$ such that $\frac{\ep}{\beta}<\alpha$, for all $n\geq n(\ep)$ and all $R^n\in\rrr^n$, $\{X^n_1\geq\alpha\}\cap \{Y^n_1=d_{R^n}\}=\emptyset$. 

Now we will look at the up parts. Assume first that $U>1$. Then, by (ii) (c), $\limsup\Pi^n<1$, hence there exists $0<\gamma\leq 1$ and $n_1\geq n_0$ such that, for all $n\geq n_1$, $\Pi^n\leq 1-\gamma$. Then, for $\ep$ with $\frac{\ep}{\beta}<\alpha$, for all $n\geq\max(n_1, n(\ep))$, we have that
$$\sup_{R^n\in\rrr^n}R^n(X^n_1\geq\alpha)\leq \sup_{R^n\in\rrr^n}R^n(Y^n_1=u_{R^n})=\sup_{R^n\in\rrr^n}\pi_{R^n}\leq \Pi^n\leq 1-\gamma<1,$$
therefore (ii) cannot hold for large $n$.

Assume now that $U=1$. Then, for all $R\in\rrr^n$, we have that $u_{R^n}\leq U^n$. As $U=1$, there exists $n_2\geq n_0$ such that, for all $n\geq n_2$, $U^n-1< \alpha\beta $, and therefore, by (\ref{Hn+1A2}),
$$X^n_1\leq H^n(U^n-1)<\frac{\alpha\beta}{\beta}=\alpha,\quad\rrr^n\text{-q.s.},$$
for all $n\geq n_2$, therefore (ii) cannot hold for large $n$. 

Case (2): as $d=D=1$, by Assumption~\ref{avoid_riskless_limit} (i), we have that $U>1$. Then, again by (ii)(c), $\limsup\Pi^n<1$ and by (ii)(b) $\frac{\max(D^n-1,0)}{1-d^n}\to0$ There exists $0<\gamma\leq 1$ and $n_0$ such that, for all $n\geq n_0$, $\Pi^n\leq 1-\gamma$. Moreover, for each $\ep>0$, there exists $n(\ep)\geq n_0$ such that $\frac{\max(D^n-1,0)}{1-d^n}<\ep$, for all $n\geq n(\ep)$. Then, for all $n\geq n_0$, we have that
$$\sup_{R^n\in\rrr^n}R^n(Y^n_1=u_{R^n})\leq \Pi^n\leq 1-\gamma<1,$$
therefore the probability of the up parts is not enough to satisfy (ii). For the down parts, observe that, for all $R^n\in\rrr^n$, we have that $d_{R^n}\leq D^n$. Choose $\ep<\alpha$, then, for all $n\geq n(\ep)$, we have that  
$$\{X^n_1\geq\alpha\}\cap \{Y^n_1=d_{R^n}\}=\emptyset,$$
as $H^n(d_{R^n}-1)\leq H^n(D^n-1)\leq \frac{\max(D^n-1,0)}{1-d^n}<\ep<\alpha$, as $H^n\leq \frac1{1-d^n}$ by (\ref{Hn+1A2}).
Hence there cannot exist $P^n\in\rrr^n$ as in (ii) for large $n\in N_1$.

Let now $n\in N_2$. The proof will be very similar to the proof for $n\in N_1$.
As $H^n<0$ the smallest possible value of $X^n_1$ appears for $R^n\in\rrr^n$ with $u_{R^n}=U^n$. 
Then $R^n(X^n_1=-|H^n|(U^n-1))=\pi_{R^n}\geq \pi^n>0$ by Assumption~\ref{ass_blan_car_1step}.
Hence, it has to hold that $-|H^n|(U^n-1)\geq-1$ and thus,  
\begin{equation}\label{Hn-1A2}
 0<|H^n|\leq\frac{1}{U^n-1}. 
\end{equation}
We will distinguish the following cases: (1) $1\leq u<U<+\infty$ and (2) $u=U=1$.

Case (1): As $U>1$ there exists $0<\beta<1$ and $n_0$ such that, for all $n\geq n_0$, $U^n\geq 1+\beta$. As $u\geq 1$, for each $\ep>0$ there exists $n(\ep)\geq n_0$, such that, for all $n\geq n(\ep)$,  we have that $u^n\geq 1-\ep$. Observe that, for each $R^n\in\rrr^n$, with $u_{R^n}\in[u^n,U^n]$, for all $n\geq n(\ep)$, we have that 
$$\{Y^n_1=u_{R^n}\}\subseteq \{X^n_1\leq \frac{\ep}{\beta}\},$$
as 
$$H^n(u_{R^n}-1)=|H^n|(-u_{R^n}+1)\leq |H^n|(-u^n+1)\leq\frac{\ep}{\beta},$$ by (\ref{Hn-1A2}) because 
$-u^n+1\leq\ep$ and $U^n-1\geq\beta$.  

Hence, (ii) cannot happen on the up parts, as, for $\ep$ such that $\frac{\ep}{\beta}<\alpha$, as for all $n\geq n(\ep)$ and all $R^n\in\rrr^n$, $\{X^n_1\geq\alpha\}\cap \{Y^n_1=u_{R^n}\}=\emptyset$. 

Now we will look at the down parts. Assume first that $d<1$. Then, by (ii)(a), $\liminf\pi^n>0$, hence there exists $0<\gamma\leq 1$ and $n_1\geq n_0$ such that, for all $n\geq n_1$, $\pi^n\geq \gamma$. Then, for $\ep$ with $\frac{\ep}{\beta}<\alpha$, for all $n\geq\max(n_1, n(\ep))$, we have that
$$\sup_{R^n\in\rrr^n}R^n(X^n_1\geq\alpha)\leq \sup_{R^n\in\rrr^n}R^n(Y^n_1=d_{R^n})=\sup_{R^n\in\rrr^n}(1-\pi_{R^n})\leq (1-\pi^n)\leq 1-\gamma<1,$$
therefore (ii) cannot hold for large $n$.

Assume now that $d=1$. Then, for all $R\in\rrr^n$, we have that $d_{R^n}\geq d^n$. As $d=1$, there exists $n_2\geq n_0$ such that, for all $n\geq n_2$, $1-d^n< \alpha\beta $, and therefore, by (\ref{Hn-1A2}), as $\frac{1}{U^n-1}\leq\frac1{\beta},$
$$X^n_1\leq |H^n|(1-d^n)<\frac{\alpha\beta}{\beta}=\alpha,\quad\rrr^n\text{-q.s.},$$
for all $n\geq n_2$, therefore (ii) cannot hold for large $n$. 

Case (2): as $u=U=1$, by Assumption~\ref{avoid_riskless_limit} (i), we have that $d<1$. Then, again by (ii)(a), $\liminf\pi^n>0$ and by (ii)(d) $\frac{\max(1-u^n,0)}{U^n-1}\to 0$. There exists $0<\gamma\leq 1$ and $n_0$ such that, for all $n\geq n_0$, $\pi^n\geq \gamma$. Moreover, for each $\ep>0$, there exists $n(\ep)\geq n_0$ such that $\frac{\max(1-u^n,0)}{U^n-1}<\ep$, for all $n\geq n(\ep)$. Then, for all $n\geq n_0$, we have that
$$\sup_{R^n\in\rrr^n}R^n(Y^n_1=d_{R^n})\leq (1-\pi^n)\leq 1-\gamma<1,$$
therefore the probability of the down parts is not enough to satisfy (ii). For the up parts, observe that, for all $R^n\in\rrr^n$, we have that $u_{R^n}\geq u^n$. Choose $\ep<\alpha$, then, for all $n\geq n(\ep)$, we have that  
$$\{X^n_1\geq\alpha\}\cap \{Y^n_1=u_{R^n}\}=\emptyset,$$
as $|H^n|(-u_{R^n}+1)\leq |H^n|(-u^n+1)\leq \frac{\max(1-u^n,0)}{U^n-1}<\ep<\alpha$, as $|H^n|\leq \frac1{U^n-1}$ by (\ref{Hn-1A2}).
Hence there cannot exist $P^n\in\rrr^n$ as in (ii) for all large $n\in N_2$.

This shows that NAA2$(\rrr^n)$ has to hold. 
 \end{proof}

\begin{proof}[Proof of Theorem~\ref{thm_U0_finite_AA2}]
    (i) $\Rightarrow$ (ii) follows by Lemma~\ref{NAA2_D_0<=1} and Lemma~\ref{NAA1+U_0<inf_u_0>=1}. Further, (a) follows by Lemma~\ref{AA2liminf_prob1}, (b) follows by Lemma~\ref{AA2liminf_prob4}, (c) follows by Lemma~\ref{AA2liminf_prob3}, (d) follows by Lemma~\ref{AA2liminf_prob2}.
    
    (ii) $\Rightarrow$ (i) follows by Lemma~\ref{NAA2_reverse_finiteU0} and Lemma~\ref{AA1_all_sets}.
         (i) $\Leftrightarrow$ (iii)   then follows by 
          by Theorem~\ref{lfmftap2}. This shows that (i), (ii), (iii) are equivalent.
         \end{proof}

\begin{lemma}\label{U0_D0_infty_AA1}
Assume that $U=+\infty$ and NAA1$(\pp^n)$ holds. Then $\limsup\Pi^n=0$ and $D<+\infty$. 
\end{lemma}

\begin{proof}[Proof of Lemma~\ref{U0_D0_infty_AA1}]
     Assume first that $\limsup\Pi^n=\beta>0$. Take a subsequence (still denoted by $n$) such that $\lim_{n\to\infty}\Pi^n=\beta$. Define $X^n_1=\frac1{U^n}(Y^n_1-1)$. Then we have that ${X}^n_1\geq -\frac1{U^n}=:-\ep_n$, $\pp^n$-q.s., with $\ep_n\to0$. Define
     $P^n\circ (Y^n_1)^{-1}=\Pi^n\delta_{U^n}+(1-\Pi^n)\delta_{D^n}$. It holds that $P^n(X^n_1\geq 1-\ep_n)\geq P^n(Y^n_1=U^n)=\Pi^n\geq\frac{\beta}2>0,$ for all large $n$. This gives an AA1, a contradiction.
     
     Assume now that $D=+\infty$. Define $X^n_1=\frac1{D^n}(Y^n_1-1)$. Then we have that ${X}^n_1\geq -\frac1{D^n}=:-\ep_n$, $\pp^n$-q.s.,with $\ep_n\to0$. Define
     $P^n\circ (Y^n_1)^{-1}=\Pi^n\delta_{U^n}+(1-\Pi^n)\delta_{D^n}$.
    As  $D^n<U^n$ by \ref{ass34},  it follows that, for large $n$,  $P^n(X^n_1\geq 1-\ep_n)=1$. This gives an AA1, even a strong asymptotic arbitrage (and thus also an AA2), a contradiction.
\end{proof}

\begin{lemma}\label{NAA1_reverse_infiniteU0}
 Let $d<1$, $D<+\infty=U$ and $\limsup\Pi^n=0$. Then NAA1$(\rrr^n)$ holds.
 \end{lemma}

\begin{proof}
     Suppose, to the contrary, that there would exist an AA1$(\rrr^n)$,  that is, there exists a subsequence of markets, still denoted by $n$, such that there exists $H^n\in\mathbb{R}$ such that for $X^n_1=H^n(Y^n_1-1)$ we have that (i) $X^n_1\geq -c_n$, $\rrr^n$-q.s., with $\lim_{n\to\infty}c_n=0$, and (ii) there exists $\alpha>0$ and $P^n\in\rrr^n$ with $P^n(X^n_1\geq\alpha)\geq\alpha$, for all $n$.      
  
     As $d<1$ and $D<+\infty$ there exists $n_0$  and $0<\beta<1$ such that, for all $n\geq n_0$, we have that $d^n< 1-\beta$ and $D^n\leq D+\frac{\alpha}2<+\infty$.  Moreover, as $c_n\to0$, $\Pi^n\to0$  and $U=+\infty$, for each $\ep>0$, there exists $n(\ep)\geq n_0$ such that for all $n\geq n(\ep)$, $\Pi^n<\ep$ and $\frac1{U^n-1}<\ep$ and $c_n\cdot \max(\frac{D+\frac{\alpha}2-1}{\beta},1)<\ep$. 
      Define, as in the proof of Lemma~\ref{NAA1_reverse_finiteU0},  $N_1=\{n\geq 1: H^n>0\}$ and $N_2=\{n\geq 1: H^n<0\}$ and, as there, assume  that $N_1$, $N_2$ both are unbounded. Let $n\in N_1$ with $n\geq n_0$. By (i), for each $R\in\rrr^n$, it has to hold that $R(X^n_1\geq-c_n)=1$. As $H^n>0$, as in the proof of Lemma~\ref{NAA1_reverse_finiteU0}, we get
\begin{equation}\label{Hn+1infty}
 0<H^n\leq\frac{c_n}{1-d^n}< \frac{c_n}{\beta}. 
\end{equation}
Further, observe that, for all $R\in \rrr^n$, and $n\geq n(\ep)$, we have that $R(X^n_1\leq H^n(D^n-1))=1-\pi_{R}> 1-\ep$ as $\pi_{R}\leq \Pi^n<\ep$. Thus, and because $D^n<D+\frac{\alpha}2$ and by (\ref{Hn+1infty}) we have that
\begin{align}\label{Hn+2infty}
 R\biggl( X^n_1\leq c_n\frac{D+\frac{\alpha}2-1}{\beta}\biggr) >1-\ep\text{, for all $R\in\rrr^n$, } &\text{if $D^n>1$}\\
 R(X^n_1\leq 0)>1-\ep\text{, for all $R\in\rrr^n$, } &\text{if $D^n\leq 1$}.\nonumber
\end{align}
Let $\ep<\alpha$. Then, for $n\geq n(\ep)$,  $c_n\cdot\max(\frac{D+\frac{\alpha}2-1}{\beta},1)<\ep<\alpha$. Hence, for  $n\in N_1$, $n\geq n(\ep)$, by (\ref{Hn+2infty})  it follows that $R(X^n_1\geq \alpha)<\ep$, for all $n\geq n(\ep)$, and there cannot exist $P^n\in\rrr^n$ as in (ii) for all large $n\in N_1$.

Let now $n\in N_2$ with $n\geq n_0$. By (i), for each $R\in\rrr^n$, it has to hold that $R(X^n_1\geq-c_n)=1$. As $H^n<0$, the smallest possible value of $X^n_1$ appears for $R^n\in\rrr^n$ with $u_{R^n}=U^n$. Then $R^n(X^n_1=-|H^n|(U^n-1))=\pi_{R^n}\geq\pi^n>0$ by Assumption~\ref{ass_blan_car_1step}. Hence it has to hold that $-|H^n|(U^n-1)\geq-c_n$ and thus, for $n\geq n(\ep)$,
\begin{equation}\label{Hn-1infty}
 0<|H^n|\leq\frac{c_n}{U^n-1}<\ep^2. 
\end{equation}
Further, observe that, $X^n_1=-|H^n|(Y^n-1)$ is $<0$ if $Y^n-1>0$ and, if $Y^n-1\leq0$, then $X^n_1=|H^n|(1-Y^n_1)\geq 0$ with maximal value $|H^n|(1-d^n)$. Hence, for all $R\in \rrr^n$, we have that $R(X^n_1\leq |H^n|(1-d^n))=1$ and therefore, as $0<d^n<1$ and by (\ref{Hn-1infty}), for all $n\geq n(\ep)$,
\begin{equation}\label{Hn-2infty}
 X^n_1\leq |H^n|(1-d^n)< \frac{c_n}{U^n-1}<\ep^2, \quad\text{$\rrr^n$-q.s.} 
\end{equation}
Let  $\ep^2<\alpha$ Then, for  $n\in N_2$, $n\geq n(\ep)$, by (\ref{Hn-2infty}), $\{X^n_1\geq \alpha\}$ is an $\rrr^n$-polar set and there cannot exist $P^n\in\rrr^n$ as in (ii) for all large $n\in N_2$. Therefore there cannot exist an AA1($\rrr^n$).
 \end{proof}

\begin{proof}[Proof of Theorem~\ref{thm_U0_infinite}]
(i) $\Rightarrow$ (ii) follows by Lemma~\ref{U0_D0_infty_AA1} and Lemma~\ref{NAA1_1>d_0}.  (ii) $\Rightarrow$ (i) follows by Lemma~\ref{NAA1_reverse_infiniteU0}.
Again, (i) $\Leftrightarrow$ (iii)  follows by Theorem~\ref{lfmftap}. This
 shows that (i), (ii), (iii) are equivalent.
\end{proof}

\begin{lemma}\label{NAA2_reverse_infiniteU0}
 Let $U=+\infty$ and assume that $\limsup{\Pi^n}<1$, $D\leq 1$ and, if $d=D=1$, then $\max(D^n-1,0)=o(1-d^n)$.  Then NAA2$(\rrr^n)$ holds.
 \end{lemma}

 \begin{proof}The first part  is almost identical to the first part of the proof of Lemma~\ref{NAA2_reverse_finiteU0}. As there
suppose, to the contrary, that there would exist an AA2$(\rrr^n)$,  that is, there exists a subsequence of markets, still denoted by $n$, such that there exists $H^n\in\mathbb{R}$ such that for $X^n_1=H^n(Y^n_1-1)$ we have that (i) $X^n_1\geq -1$, $\rrr^n$-q.s., and (ii) there exists $\alpha>0$ and $P^n\in\rrr^n$ with $P^n(X^n_1\geq\alpha)\geq1-\ep_n$,  where $\lim_{n\to\infty}\ep_n=0$.      
Define again  $N_1=\{n\geq 1: H^n>0\}$ and $N_2=\{n\geq 1: H^n<0\}$ and, as there, assume  that $N_1$, $N_2$ both are unbounded. 

Let $n\in N_1$. As in the proof of Lemma~\ref{NAA2_reverse_finiteU0} we get
$$
 0<H^n\leq\frac{1}{1-d^n}.$$ 

We will distinguish the same two  cases as before: (1) $d<D\leq1$ and (2) $d=D=1$.

Case (1): proceed as for  Case (1) in the proof of Lemma~\ref{NAA2_reverse_finiteU0} using that now $U=+\infty>1$ and, by assumption, $\limsup\Pi^n<1$. Exactly as there we see that (ii) cannot happen on the down parts. For the up parts with $U=+\infty$ we only have to proceed as for the part with $U>1$ of the  proof of Lemma~\ref{NAA2_reverse_finiteU0} and see exactly as there  that (ii) cannot hold on the up parts as well. 

Case (2): As here, by assumption, $\limsup\Pi^n<1$, the proof works exactly as in the proof of  Lemma~\ref{NAA2_reverse_finiteU0}.
Hence there cannot exist $P^n\in\rrr^n$ as in (ii) for large $n\in N_1$.

Let now $n\in N_2$. Here the proof becomes slightly easier for $U=+\infty$ because of the role of the uniform bound of -1 by going short in $S^n$. 
As $H^n<0$ the smallest possible value of $X^n_1$ appears for $R^n\in\rrr^n$ with $u_{R^n}=U^n$. Then $R^n(X^n_1=-|H^n|(U^n-1))>0$. Hence it has to hold that $-|H^n|(U^n-1)\geq-1$. As $U=+\infty$, for each $\ep>0$ there exists $n(\ep)$ such that, for all $n\geq n(\ep)$, $\frac{1}{U^n-1}<\ep$. Therefore 
$$
 0<|H^n|<\ep,\quad\text{for all $n\geq n(\ep)$}. 
$$
Observe that, for each $R\in\rrr^n$, we have that $R(Y^n_1\in\{u_{R},d_{R}\})=1$, where $\min(u_R,d_R)\geq \min(u^n,d^n)=d^n$, by \ref{ass34}. We have that
$$X^n_1=-|H^n|(Y^n_1-1)\leq \max(0, |H^n|(1-Y^n_1)).$$ Therefore, for all $n\geq n(\ep)$, and  all $R\in\rrr^n$, we have that
$$R(X^n_1\leq \ep)=1,$$
as, for each fixed $R$, $\max(0, |H^n|(1-Y^n_1))\leq \max(0, \ep(1-d^n))<\ep$. 
Hence there cannot exist $P^n\in\rrr^n$ as in (ii) for all large $n\in N_2$.
This shows that NAA2$(\rrr^n)$ has to hold. 
 \end{proof}

\begin{proof}[Proof of Theorem~\ref{thm_U0_infinite_AA2}]
    (i) $\Rightarrow$ (ii) follows by using Lemma~\ref{NAA2_D_0<=1}, Lemma~\ref{AA2liminf_prob4} and  Lemma~\ref{AA2liminf_prob3}. Note that these Lemmas do not assume $U<+\infty$.
    
    (ii) $\Rightarrow$ (i) follows by Lemma~\ref{NAA2_reverse_infiniteU0} together with Lemma~\ref{AA1_all_sets}.
    
    (i) $\Leftrightarrow$ (iii)   then follows by Theorem~\ref{lfmftap2}. This shows that (i), (ii), (iii) are equivalent.
\end{proof}

\section*{Acknowledgements}
We would like to thank Felix-Benedikt Liebrich for pointing out a gap in an earlier version of this paper, and introducing us to his work on robust $L^{1}$-$L^{\infty}$ duality theory. 
We would also like to thank Laurence Carassus for drawing our attention to some technical imprecisions in an earlier version of this paper and  to her work on robust binomial models.

\bibliographystyle{chicago}
\bibliography{main}
\end{document}